\documentclass[12pt]{amsart}
\usepackage{amsmath}
\usepackage{amscd}
\usepackage{amssymb}
\usepackage[all]{xy}
\usepackage[dvips]{graphicx}
\makeatletter
\newcommand\testshape{family=\f@family; series=\f@series; shape=\f@shape.}
\def\myemphInternal#1{\if n\f@shape%
\begingroup\itshape #1\endgroup\/%
\else\begingroup\bfseries #1\endgroup%
\fi}
\def\myemph{\futurelet\testchar\MaybeOptArgmyemph}
\def\MaybeOptArgmyemph{\ifx[\testchar \let\next\OptArgmyemph
                 \else \let\next\NoOptArgmyemph \fi \next}
\def\OptArgmyemph[#1]#2{\index{#1}\myemphInternal{#2}}
\def\NoOptArgmyemph#1{\myemphInternal{#1}}
\makeatother

\newtheorem{theorem}[subsection]{Theorem}
\newtheorem{lemma}[subsection]{Lemma}
\newtheorem{proposition}[subsection]{Proposition}
\newtheorem{corollary}[subsection]{Corollary}
\newtheorem{remark}[subsection]{Remark}
\newtheorem{example}[subsection]{Example}
\newtheorem{definition}[subsection]{Definition}
\newtheorem{claim}[subsection]{Claim}

\makeatletter
\@addtoreset{equation}{section}
\@addtoreset{figure}{section}
\@addtoreset{table}{section}
\makeatother

\providecommand\eqref[1]{(\ref{#1})}

\newcommand\CCC{{\mathbb C}}

\newcommand\RRR{{\mathbb R}}

\newcommand\ZZZ{{\mathbb Z}}

\newcommand\PathComp[2]{ {#1}_{#2} }

\newcommand\id{\mathrm{id}}

\newcommand\Int{\mathrm{Int}}

\newcommand\Orbit{\mathcal{O}}
\newcommand\Stab{\mathcal{S}}
\newcommand\Diff{\mathcal{D}}

\newcommand\Cinf{C^{\infty}}

\newcommand\Unity[1]{ \PathComp{#1}{\id} }

\newcommand\StabId{\Unity{\Stab}}

\newcommand\aCircle{S^1}
\newcommand\Psp{P}
\newcommand\tnmanif{\widetilde{\nmanif}} 

\newcommand\mrsfunc{f}

\newcommand\DiffM{\Diff(\Mman)}

\newcommand\difM{h}

\newcommand\singf{\Sigma_{\func}}

\newcommand\Orbf{\Orbit(\mrsfunc)}
\newcommand\Orbff{\Orbit_{\mrsfunc}(\mrsfunc)}

\newcommand\DiffIdM{\Diff_{\id}(\Mman)}

\newcommand\Stabf{\Stab(\mrsfunc)}
\newcommand\StabIdf{\Stab_{\id}(\mrsfunc)}

\newcommand\tdifM{\tilde{\difM}}
\newcommand\bdifM{\bar{\difM}}

\newcommand\nmanif{N}

\newcommand\DiffId{\Diff_{\id}}

\newcommand\torus{\mathbb{T}^2}

\newcommand\prjplane{\RRR\mathrm{P}^2}

\newcommand\crcomp{K}

\newcommand\func{f}

\newcommand\Mman{M}

\newcommand\Func{F}

\newcommand\tr{\mathrm{tr}}

\newcommand\regnbh{R}
\newcommand\cannbh{N}

\newcommand\regnbha[1]{\regnbh_{#1}}
\newcommand\cannbha[1]{\cannbh_{#1}}

\newcommand\cwpart{\Xi}
\newcommand\tcwpart{\widetilde{\cwpart}}

\newcommand\chains{C}

\newcommand\fc{k}

\newcommand\regnbhg{R_{\kgraph}}
\newcommand\cannbhg{N_{\kgraph}}
\newcommand\kgraph{K}

\newcommand\partf{\Delta_{\mrsfunc}}
\newcommand\partfreg{\partf^{\mathrm{reg}}}
\newcommand\partfcr{\partf^{\mathrm{cr}}}

\newcommand\afunc{\alpha}

\newcommand\cwpartg{\cwpart_{\kgraph}}

\newcommand\hdifM{\widehat{\difM}}

\newcommand\manifneg{X}
\newcommand\StabAf{\Stab'(\mrsfunc)}
\newcommand\StabAfneg{\Stab'(\mrsfunc,\manifneg)}

\newcommand\StabAfBi{\Stab'(\mrsfunc|_{B_i},\partial B_i)}

\newcommand\StabIdAfneg{\StabId'(\mrsfunc,\manifneg)}
\newcommand\StabIdAf{\StabId'(\mrsfunc)}

\newcommand\Hdif{H}
\newcommand\Gdif{G}
\newcommand\tHdif{\widetilde{\Hdif}}

\newcommand\dif{h}
\newcommand\gdif{g}
\newcommand\tMman{\widetilde{\Mman}}
\newcommand\Mmanneg{X}
\newcommand\AFld{F}

\newcommand\AFlow{\mathbf{\AFld}}

\newcommand\Cont[1]{\mathcal{C}^{#1}}
\newcommand\Cr[3]{\Cont{#1}(#2,#3)}
\newcommand\Ci[2]{\Cr{\infty}{#1}{#2}}

\newcommand\Wr[1]{\Cont{#1}}

\newcommand\AxBd{{\sf(Bd)}}
\newcommand\AxSPN{{\sf(SPN)}}
\newcommand\AxFibr{{\sf(Fibr)}}

\newcommand\AxAdm{{\sf(SA)}}

\newcommand\AxIsol{{\sf(Isol)}}

\newcommand\ST{\mathsf{S}}
\newcommand\PT{\mathsf{P}}
\newcommand\NT{\mathsf{N}}

\newcommand\Uman{U}
\newcommand\Vman{V}
\newcommand\Wman{W}
\newcommand\orb{o}

\newcommand\HFld{H}
\newcommand\HFlow{\mathbf{\HFld}}

\newcommand\hcannbhg{\widehat{N}_{\kgraph}}

\newcommand\regnbhAll{\mathbf{R}}

\newcommand\regnbhNDn{\mathcal{R}_{<0}}

\newcommand\Kcrcomp{\mathcal{K}}
\begin{document}

\author{Sergiy Maksymenko}
\title
[Functions on surfaces and incompressible subsurfaces]
{Functions on surfaces and \\ incompressible subsurfaces}
\address{Topology dept., Institute of Mathematics of NAS of Ukraine, Te\-re\-shchenkivska st. 3, Kyiv, 01601 Ukraine}
\email{maks@imath.kiev.ua}
\date{08.01.2010}

\thanks{This research is partially supported by grant of Ministry of Science and Education of Ukraine, No. M/150-2009}

\begin{abstract}
Let $\Mman$ be a smooth connected compact surface, $\Psp$ be either a real line $\RRR$ or a circle $\aCircle$.
Then we have a natural \myemph{right} action of the group $\DiffM$ of diffeomorphisms of $\Mman$ on $\Ci{\Mman}{\Psp}$.
For $\func\in\Ci{\Mman}{\Psp}$ denote respectively by $\Stabf$ and $\Orbf$ its stabilizer and orbit with respect to this action.
Recently, for a large class of smooth maps $\func:\Mman\to\Psp$ the author calculated the homotopy types of the connected components of $\Stabf$ and $\Orbf$.
It turned out that except for few cases the identity component of $\Stabf$ is contractible, $\pi_i\Orbf=\pi_i\Mman$ for $i\geq3$, and $\pi_2\Orbf=0$, while $\pi_1\Orbf$ it only proved to be a finite extension of $\pi_1\DiffIdM\oplus\ZZZ^{l}$ for some $l\geq0$.
In this note it is shown that if $\chi(\Mman)<0$, then $\pi_1\Orbf=G_1\times\cdots\times G_n$, where each $G_i$ is a fundamental group of the restriction of $\func$ to a subsurface $B_i\subset\Mman$ being either a $2$-disk or a cylinder or a M\"obius band.
For the proof of main result incompressible subsurfaces and cellular automorphisms of surfaces are studied.
\end{abstract}

\maketitle

\section{Introduction}
Let $\Mman$ be a smooth compact connected surface and $\Psp$ be either the real line $\RRR$ or the circle $\aCircle$.
Consider the \myemph{right} action of the group $\DiffM$ of diffeomorphisms of $\Mman$ on $\Ci{\Mman}{\Psp}$ defined by 
$$ 
\difM\cdot \func = \func \circ \difM^{-1}
$$ 
for $\difM\in\DiffM$ and $\func\in\Ci{\Mman}{\Psp}$.
For every $\func\in\Ci{\Mman}{\Psp}$ let 
$$\Orbf=\{\func \circ \difM \ | \ \difM \in \DiffM \},$$ 
$$\Stabf=\{\difM \ | \ \func = \func \circ \difM, \ \difM\in\DiffM\}$$
be respectively the orbit and the stabilizer of $\func$ with respect to this action.
We will endow $\DiffM$, $\Stabf$, $\Ci{\Mman}{\Psp}$, and $\Orbf$ with the corresponding topologies $\Wr{\infty}$.
Denote by $\StabIdf$ the identity path component of $\Stabf$ and by $\Orbff$ the path component of $\func$ in $\Orbf$.
In\;\cite{Maks:AGAG:2006} the author calculated the homotopy types of $\StabIdf$ and $\Orbff$ for all Morse maps $\func:\Mman\to\Psp$.

Moreover, in\;\cite{Maks:func-isol-sing} the results of\;\cite{Maks:AGAG:2006} were extended to a large class of maps with (even degenerate) isolated critical points satisfying certain ``non-degeneracy'' conditions.
In fact there were introduced three types of isolated critical points (called $\ST$, $\PT$, and $\NT$) and the following three axioms for $\func$:
\begin{enumerate}
 \item[\AxBd]
\myemph{$\func$ takes constant value at each connected component of $\partial\Mman$ and $\singf\subset\Int{\Mman}$.}
 \item[\AxSPN]
\myemph{Every critical point of $\func$ is either an $\ST$- or a $\PT$- or an $\NT$-point.}
 \item[\AxFibr]
\myemph{The natural map $p:\DiffM \to \Orbf$ defined by $p(\dif) = \func\circ \dif^{-1}$ 
is a Serre fibration with fiber $\Stabf$ in topologies $\Wr{\infty}$.}
\end{enumerate}

Recall that if $\func:(\CCC,0)\to(\RRR,0)$ is a smooth germ for which $0\in\CCC$ is an \emph{isolated} critical point, then there exists a \emph{homeomorphism} $\difM:\CCC\to\CCC$ such that $\difM(0)=0$ and
$$
\func\circ\difM(z) = 
\left\{
\begin{array}{ll}
\pm|z|^2, & \text{if $z$ is a \myemph{local extremum}, \cite{Dancer:2:JRAM:1987}},  \\
\text{Re}(z^n), (n\geq1) & \text{otherwise, so $z$ is a \myemph{saddle}, \cite{Prishlyak:TA:2002}},
\end{array}
\right.
$$
Examples of the foliation by level sets of $\func$ near $0$ are presented in Figure\;\ref{fig:isol_crit_pt}.
\begin{figure}[ht]
\begin{tabular}{ccccccc}
\includegraphics[width=0.18\textwidth]{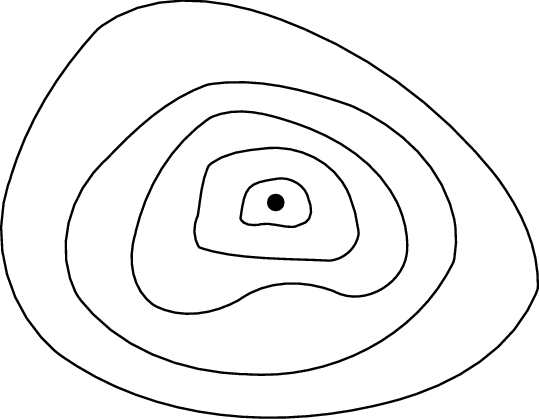}  & 
\quad 
\includegraphics[width=0.18\textwidth]{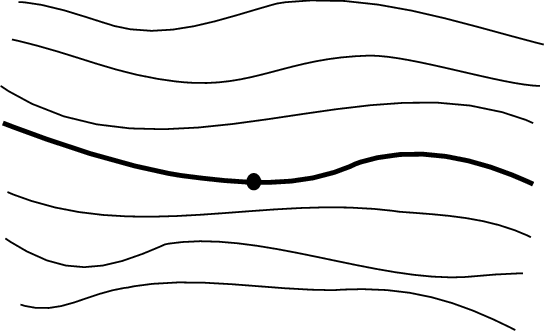}  & 
\quad 
\includegraphics[width=0.18\textwidth]{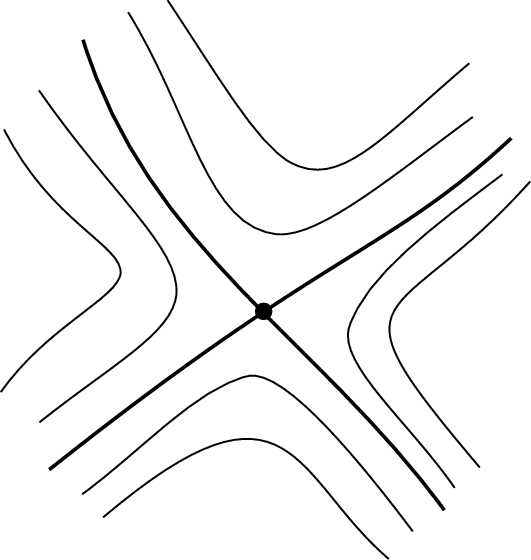}  & 
\quad 
\includegraphics[width=0.18\textwidth]{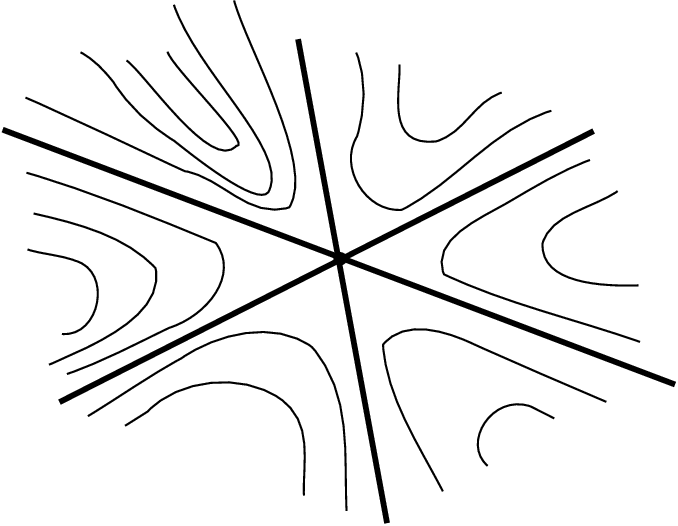} 
\end{tabular}
\caption{Isolated critical points}\label{fig:isol_crit_pt}
\end{figure}

From this point of view $\ST$-points are saddles, while $\PT$- and $\NT$-points a local extremes.
Moreover, $\PT$-points admit non-trivial $\func$-preserving circle actions (as non-degenerate local extremes do), while $\NT$-points admit only $\ZZZ_n$-action preserving $\func$.
We will not give precise definitions but recall a large class of examples of such points.
\begin{example}\label{exmp:SPN-points}{\rm\cite{Maks:AGAG:2006}.}
Let $\func:\RRR^2\to\RRR$ be a homogeneous polynomial without multiple factors with $\deg\func\geq2$, so
$$
\func = L_1\cdots L_a \cdot Q_1 \cdots Q_b, \qquad a+2b\geq2,
$$
where every $L_i$ is a linear function and every $Q_j$ is an irreducible over $\RRR$ (i.e. definite) quadratic form such that $L_i/L_{i'}\not=\mathrm{const}$ for $i\not=i'$ and $Q_j/Q_{j'}\not=\mathrm{const}$ for $j\not=j'$.

If $a\geq1$, so $\func$ has linear factors and thus $0$ is a saddle, then the origin $0\in\RRR^2$ is an $\ST$-point for $\func$.

If $a=0$ and $b=1$, so $\func=Q_1$, then the origin $0\in\RRR^2$ is a $\PT$-point for $\func$.

Otherwise, $a=0$ and $b\geq2$, so $\func=Q_1\cdots Q_b$.
Then the origin $0\in\RRR^2$ is an $\NT$-point for $\func$.
\end{example}
\begin{lemma}{\rm\cite{Maks:AGAG:2006}.}
Let $\func:\Mman\to\Psp$ be a $\Cinf$ map satisfying \AxBd, and such that every of its critical points belongs to the class described in Example\;\ref{exmp:SPN-points}, in particular, $\func$ also satisfies \AxSPN.
Then $\func$ also satisfies \AxFibr.
\end{lemma}

It follows from Morse lemma and Example\;\ref{exmp:SPN-points} that non-degenerate saddles are $\ST$-points while non-degenerate local extremes are $\PT$-points.

Now the main result of \cite{Maks:func-isol-sing} can be formulated as follows.
\begin{theorem}\label{th:StabIdf_Orbff_A123}{\rm\cite{Maks:AGAG:2006, Maks:func-isol-sing}.}
Suppose $\func:\Mman\to\Psp$ satisfies \AxBd\ and \AxSPN.
If $\func$ has at least one $\ST$- or $\NT$-point, or if $\Mman$ is non-orientable, then $\StabIdf$ is contractible.

Moreover, if in addition $\func$ satisfies \AxFibr, then $\pi_i\Orbff=\pi_i\Mman$ for $i\geq3$, $\pi_2\Orbff=0$, and for $\pi_1\Orbf$ we have the following short exact sequence 
$$
1 \to \pi_1\DiffM\oplus\ZZZ^{l} \to \pi_1\Orbff \to G \to 1,
$$
for a certain finite group $G$ and $l\geq0$ both depending on $\func$.
\end{theorem}

Thus, the information about the fundamental group $\pi_1\Orbff$ is not complete.
The aim of this note is to show that the calculation of $\pi_1\Orbff$ can be reduced to the case when $\Mman$ is either a $2$-disk, or a cylinder, or a M\"obius band, see Theorems\;\ref{th:stab1-repr} and\;\ref{th:pi1orb_decomp} below.
The obtained results hold for a more general class of maps $\Mman\to\Psp$ than the one considered in\;\cite{Maks:func-isol-sing}.

\subsection{Admissible critical points.}
We will now introduce a certain type of critical points for $\func$.
Let $\AFld$ be a vector field on $\Mman$, $\Vman\subset\Mman$ be an open subset, and $\dif:\Vman\to\Mman$ be an embedding.
Say that $\dif$ \myemph{preserves orbits of $\AFld$} if for every orbit $\orb$ of $\AFld$ we have that $\dif(\Vman\cap\orb)\subset\orb$.

\begin{definition}\label{defn:admissible_cr_pt}
Let $\func:\Mman\to\Psp$ be a $\Cinf$ map and $z\in\Int\Mman$ be an isolated critical point of $\func$ which is not a local extreme (so $z$ is a saddle).
Say that $z$ is \myemph{admissible} if there exists a neighbourhood $\Uman$ of $z$ containing no other critical points of $\func$ and a vector field $\AFld$ on $\Uman$ having the following properties:
\begin{enumerate}
 \item[\rm(1)]
$\func$ is constant along orbits of $\AFld$ and $z$ is a unique singular point of $\AFld$.
 \item[\rm(2)]
Let $(\AFlow_{t})$ be the local flow of $\AFld$ on $\Uman$.
Then for every germ of diffeomorphisms $\dif:(\Mman,z)\to(\Mman,z)$ preserving orbits of $\AFld$ there exists a $\Cinf$ germ $\sigma:(\Mman,z)\to\RRR$ such that $\dif(x)=\AFlow(x,\sigma(x))$ near $z$.
\end{enumerate}
\end{definition}

This definition almost coincides with the definition of an $\ST$-point, c.f.\;\cite{Maks:func-isol-sing}.
The difference is that for $\ST$-points it is also required that the correspondence $\dif\mapsto\sigma$ is continuous with respect to topologies $\Cinf$.
In particular every $\ST$-point is admissible.

Now put the following two axioms for $\func$ both implied by \AxSPN:
\begin{itemize}
 \item[\AxIsol]
\myemph{All critical points of $\func$ are isolated.}

\smallskip 

 \item[\AxAdm]
\myemph{Every saddle of $\func$ is admissible.}
\end{itemize}

\subsection{Main result}
Let $\DiffIdM$ be the identity path component of the group $\DiffM$ and 
$$\StabAf=\Stabf\cap\DiffIdM$$
be the stabilizer of $\func$ with respect to the right action of $\DiffIdM$.
Thus $\StabAf$ consists of diffeomorphisms $\dif$ isotopic to $\id_{\Mman}$ and preserving $\Func$, i.e.\!  $\func\circ\dif=\func$.

For a closed subset $\Mmanneg\subset\Mman$ denote by $\StabAfneg$ the subgroup of $\StabAf$ consisting of diffeomorphisms fixed on some neighbourhood of $\Mmanneg$.

The aim of this note is to prove the following theorem:
\begin{theorem}\label{th:stab1-repr}
Suppose $\chi(\Mman)<0$.
Let $\func:\Mman\to\Psp$ be a $\Cinf$ map satisfying the axioms \AxBd, \AxIsol, and \AxAdm.
Then there exists a compact subsurface $\Mmanneg\subset\Mman$ with the following properties:

{\rm(1)}~$\func$ is locally constant on $\partial\Mmanneg$ and every connected component $B$ of $\overline{\Mman\setminus\Mmanneg}$ is either a $2$-disk or a $2$-cylinder or a M\"obius band.
Moreover, $\partial B\subset\Mmanneg$ and $B$ contains critical points of $\func$.

{\rm(2)}~Let $\difM\in\StabAfneg$ and $B$ be a connected component of $\overline{\Mman\setminus\Mmanneg}$, thus $\difM$ is fixed on some neighbourhood of $\partial B$.
Then the restriction $\difM|_{B}$ is isotopic in $B$ to $\id_{B}$ with respect to some neighbourhood of $\partial B$.

{\rm(3)}~The inclusion $i:\StabAfneg\subset\StabAf$ induces a group isomorphism $i_0:\pi_0\StabAfneg\approx\pi_0\StabAf.$
\end{theorem}
The proof of this theorem will be given in\;\S\ref{sect:proof_th:stab1-repr}.
We will now show how to simplify calculations of $\pi_1\Orbf$ using Theorem\;\ref{th:stab1-repr}.

Let $\Mmanneg$ be the surface of Theorem~\ref{th:stab1-repr} and let $B_1,\ldots,B_l$ be all the connected components of $\overline{\Mman\setminus\Mmanneg}$.
For every $i=1,\ldots,l$ denote by $\DiffId(B_i,\partial B_i)$ the group of diffeomorphisms of $B_i$ fixed on some neighbourhood of $\partial B_i$ and isotopic to $\id_{B_i}$ relatively to some neighbourhood of $B_i$.
Let also $\StabAfBi$ be the stabilizer of the restriction $\func|_{B_i}:B_i\to\Psp$ with respect to the right action of $\DiffId(B_i,\partial B_i)$.
Then we have an evident isomorphism of groups:
\begin{equation}\label{equ:StabAfneg-prod}
\psi:\StabAfneg \; \approx \; \mathop\times\limits_{i=1}^{l} \StabAfBi,
\qquad 
\psi(\difM) = (\difM|_{B_1},\ldots, \difM|_{B_l}),
\end{equation}
It is easy to show that $\psi$ is in fact a homeomorphism with respect to the corresponding $\Cinf$ topologies.

\begin{theorem}\label{th:pi1orb_decomp}
Under assumptions of Theorem~\ref{th:stab1-repr} suppose that $\func$ also satisfies \AxFibr.
Then we have an isomorphism:
$$
\pi_1\Orbff \ \approx \ \mathop\times\limits_{i=1}^{l} \pi_0\StabAfBi.
$$
\end{theorem}
\begin{proof}
It is easy to show that if $\func$ satisfies \AxFibr, then $\Orbff$ is the orbit of $\func$ with respect to the action of $\DiffIdM$ and the projection $p:\DiffIdM\to\Orbff$ is a Serre fibration as well, see\;\cite{Maks:TrMath:2008}.
Hence we get the following part of exact sequence of homotopy groups
$$
\cdots \to \pi_1 \DiffIdM \to \pi_1 \Orbff \to \pi_0\StabAf \to \pi_0\DiffIdM \to \cdots 
$$
Since $\chi(\Mman)<0$, we have $\pi_1 \DiffIdM=0$, \cite{EarleSchatz:DG:1970, EarleEells:DG:1970, Gramain:ASENS:1973}.
Moreover, $\DiffIdM$ is path-connected, whence together with Theorem~\ref{th:stab1-repr} we obtain an isomorphism:
$$
\pi_1\Orbff \;\approx\;
\pi_0\StabAf \;\stackrel{i_0}{\approx}\;
\pi_0\StabAfneg \;\stackrel{\eqref{equ:StabAfneg-prod}}{\approx}\;
\mathop\times\limits_{i=1}^{l} \pi_0\StabAfBi.
$$
Theorem is proved.
\end{proof}
Thus a general problem of calculation of $\pi_1\Orbff$ for maps satisfying the above axioms completely reduces to the case when $\chi(\Mman)\geq0$.
A presentation for $\pi_1\Orbff$ will be given in another paper.

\subsection{Structure of the paper.}
In next four sections we study incompressible subsurfaces $\nmanif\subset\Mman$.
\S\ref{sect:incompr-subsurfaces} contains their definition and some elementary properties.
In\;\S\ref{sect:incompr-surf-associated-to-maps} we show how such subsurfaces appear in studying maps $\Mman\to\Psp$ with isolated singularities.
In \S\ref{sect:deform-incompr-surf} and \S\ref{sect:cell-auto} we extend results of W.\;Jaco and P.\;Shalen\;\cite{JacoShalen:Topology:1977} about deformations of incompressible subsurfaces and periodic automorphisms of surfaces.
\S\ref{sect:deform-diff-crit-comp} contains two technical statements about deformations of diffeomorphisms preserving a map $\Mman\to\Psp$.
Finally in\;\S\ref{sect:proof_th:stab1-repr} we prove Theorem\;\ref{th:stab1-repr}.

\section{Incompressible subsurfaces}\label{sect:incompr-subsurfaces}
The following Lemma\;\ref{lm:charact-incompr} is well-known, see e.g.\;\cite[Pr.\;2.1]{ParisRolfsen:AIF:1999}.
It was also implicitly formulated in~\cite[page 359]{JacoShalen:Topology:1977}.
\begin{lemma}\label{lm:charact-incompr}
{\rm1)}~Let $\Mman$ be a connected surface, and $\nmanif\subset\Int\Mman$ be a proper compact (possibly not connected) subsurface neither of whose connected components is a $2$-disk.
Then the following conditions are equivalent:
\begin{enumerate}
\item[\rm(a)]
for every connected component $\nmanif_i$ of $\nmanif$ the inclusion homomorphism $\pi_1\nmanif_i\to\pi_1\Mman$ is injective;
\item[\rm(b)]
none of the connected components of $\overline{\Mman\setminus\nmanif}$ is a $2$-disk.
\end{enumerate}

If these conditions hold, then $\nmanif$ will be called \myemph{incompressible}, see~\cite[Def.~3.2]{JacoShalen:Topology:1977}.
\end{lemma}

\begin{corollary}\label{cor:incompr_chiM_chiN}
If $\nmanif\subset\Mman$ is incompressible, then $\chi(\Mman)\leq \chi(\nmanif)$.
\end{corollary}

\begin{corollary}\label{cor:incompr_pi1_incl_homomorphism}
Let $\regnbh\subset\Int\Mman$ be a proper compact connected subsurface.
Then the following conditions are equivalent:
\begin{enumerate}
 \item[\rm(R1)]
the homomorphism $\xi:\pi_1\regnbh\to\pi_1\Mman$ is trivial;
 \item[\rm(R2)]
$\regnbh$ is contained in some $2$-disk $D\subset\Mman$.
\end{enumerate}
\end{corollary}
\begin{proof}
The implication (R2)$\Rightarrow$(R1) is evident.

(R1)$\Rightarrow$(R2).
Suppose $\regnbh$ is not contained in any $2$-disk.
We will show that $\xi$ is non-trivial.
Let $\cannbh$ be the union of $\regnbh$ with all of the connected components of $\overline{\Mman\setminus\cannbh}$ which are $2$-disks.
Then by our assumption $\cannbh$ is not a $2$-disk and by Lemma\;\ref{lm:charact-incompr} $\cannbh$ is incompressible.
Notice that $\xi$ is a product of homomorphisms induced by the inclusions $\regnbh\subset\cannbh\subset\Mman$:
$$\xi=\beta\circ\alpha:\pi_1\regnbh \stackrel{\alpha}{\to} \pi_1\cannbh \stackrel{\beta}{\to} \pi_1\Mman.$$
Also notice that $\alpha$ is surjective and by Lemma\;\ref{lm:charact-incompr} $\beta$ is a non-trivial monomorphism.
Hence $\xi$ is also non-trivial.
\end{proof}

\begin{corollary}\label{cor:incompr_complement}
Let $\regnbh\subset\Int\Mman$ be a proper (possibly non connected) subsurface such that neither of its connected components is contained in some $2$-disk.
Then every connected component $B$ of $\overline{\Mman\setminus\regnbh}$ which is not a $2$-disk is incompressible.
\end{corollary}
\begin{proof}
Let $C$ be a connected component of $\overline{\Mman\setminus B}$.
Due to Lemma\;\ref{lm:charact-incompr} it suffices to show that $C$ is not a $2$-disk.
Notice that $C\cap\regnbh\not=\varnothing$, whence it contains some connected component $\regnbh_i$ of $\regnbh$.
By Corollary\;\ref{cor:incompr_pi1_incl_homomorphism} the product of homomorphisms $\pi_1\regnbh_i\to\pi_1 C \to \pi_1\Mman$ is non-trivial, and therefore $\pi_1 C \to \pi_1\Mman$ is also non-trivial.
This implies that $C$ is not a $2$-disk.
\end{proof}

\section{Incompressible subsurfaces associated to a map $\Mman\to\Psp$}\label{sect:incompr-surf-associated-to-maps}
\subsection{Singular foliation $\partf$ of $\func$.}\label{sect:singular-foliat}
Let $\func:\Mman\to\Psp$ be a map satisfying axioms \AxBd\ and \AxIsol.
Then $\func$ induces on $\Mman$ a one-dimen\-sional foliation $\partf$ with singularities defined as follows: \myemph{a subset $\omega\subset\Mman$ is a leaf of $\partf$ if and only if $\omega$ is either a critical point of $\func$ or a connected component of the set $\func^{-1}(c)\setminus\singf$ for some $c\in\Psp$}.
Thus the leaves of $\partf$ are $1$-dimensional submanifolds of $\Mman$ and critical points of $\func$.
Local structure of $\partf$ near critical points of $\func$ is illustrated in Figure\;\ref{fig:isol_crit_pt}.

Denote by $\partfreg$ the union of all leaves of $\partf$ homeomorphic to the circle and by $\partfcr$ the union of all other leaves.
The leaves in $\partfreg$ (resp. $\partfcr$) will be called \myemph{regular} (resp. \myemph{critical}).
Similarly, connected components of $\partfreg$ (resp. $\partfcr$) will be called \myemph{regular} (resp. \myemph{critical}) components of $\partf$.
It follows from \AxBd\ that $\partial\Mman\subset\partfreg$.
It is also evident, that every critical leaf of $\partfcr$ either is homeomorphic to an open interval or is a critical point of $\func$.

\subsection{Atoms and canonical neighbourhoods of critical components of $\partf$.}\label{sect:atoms}
For every critical component $\crcomp$ of $\partf$ define its regular neighbourhood $\regnbhg$ as follows.
Let $c_1,\ldots,c_l$ be all the critical values of $\func$ and the values of $\func$ on $\partial\Mman$.
Since $\Mman$ is compact, it follows from axioms \AxBd\ and \AxIsol\ that $l$ is finite.
For each $i=1,\ldots,l$ let $\Wman_i\subset\Psp$ be a closed connected neighbourhood (i.e. just an arc) of $c_i$ containing no other $c_j$.
We will assume that $\Wman_i\cap\Wman_j=\varnothing$ for $i\not=j$.

Now let $\crcomp$ be a critical component of $\partf$.
Then $\func(\crcomp)=c_i$ for some $i$.
Let $\regnbha{\crcomp}$ be the connected component of $\func^{-1}(\Wman_i)$ containing $\crcomp$.
Evidently, $\regnbha{\crcomp}$ is a union of leaves of $\partf$.
Following~\cite{BolsinovFomenko:1997} we will call $\regnbha{\crcomp}$ an \myemph{atom} of $\crcomp$, see Figure\;\ref{fig:atom}.
\begin{figure}[ht]
\includegraphics[height=2cm]{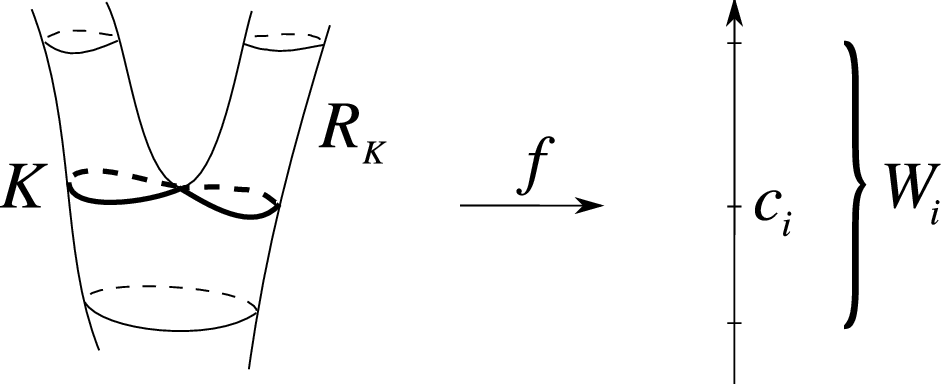}
\caption{}\protect\label{fig:atom}
\end{figure}

Evidently, $\regnbha{\crcomp}$ is a regular neighbourhood of $\crcomp$ with respect to some triangulation of $\Mman$.
Similarly to\;\cite{JacoShalen:Topology:1977} define the \myemph{canonical neighbourhood} $\cannbha{\crcomp}$ of $\crcomp$ to be the union of $\regnbha{\crcomp}$ with all the connected components of $\overline{\Mman\setminus\regnbha{\crcomp}}$ being $2$-disks.
If $\cannbha{\crcomp}$ is not a $2$-disk, then by Lemma~\ref{lm:charact-incompr} $\cannbha{\crcomp}$ is incompressible in $\Mman$.

Notice that 
\begin{equation}\label{equ:partial_regnbh}
\partial\regnbha{\crcomp} \ = \ \func^{-1}(\partial\Wman_i) \ \cap \ \regnbha{\crcomp}.
\end{equation}
Let $\crcomp'$ be another critical component of $\partf$ such that $\func(\crcomp')=\func(\crcomp)$.
Since $\regnbha{\crcomp'}$ is also constructed via $\Wman_i$, we obtain from\;\eqref{equ:partial_regnbh} that \myemph{$\func$ takes on $\partial\regnbha{\crcomp'}$ the same values as on $\partial\regnbha{\crcomp}$.}
This technical assumption is not essential, however it will be useful for the proof of Theorem\;\ref{th:stab1-repr}.

\begin{lemma}\label{lm:properties_of_atoms}
Let $\crcomp$ and $\crcomp'$ be two distinct critical components of $\partf$.
\begin{enumerate}
 \item[(i)]
Then $\regnbha{\crcomp}\cap\regnbha{\crcomp'}=\varnothing$, while $\cannbha{\crcomp}$ and $\cannbha{\crcomp'}$ are either disjoint or one of them, say $\cannbha{\crcomp}$, is contained in $\cannbha{\crcomp'}$.
In the last case $\cannbha{\crcomp}$ is a $2$-disk.

 \item[(ii)] 
Suppose $\func(\crcomp)=\func(\crcomp')$ and there exists $\dif\in\Stabf$ such that $\dif(\crcomp)=\crcomp'$.
Then $\dif(\regnbha{\crcomp})=\regnbha{\crcomp'}$ and $\dif(\cannbha{\crcomp})=\cannbha{\crcomp'}$.
\end{enumerate}
\end{lemma}
\begin{proof}
(i) follows from the assumption that $\Wman_i\cap\Wman_j=\varnothing$ for $i\not=j$, and (ii) follows from\;\eqref{equ:partial_regnbh}.
We leave the details for the reader.
\end{proof}

\begin{lemma}\label{lm:crcomp-2disk}
Let $\crcomp$ be a critical component of $\partf$ such that $\cannbha{\crcomp}$ is a $2$-disk.
Then either 
\begin{enumerate}
\item[(i)] 
$\Mman$ is a $2$-disk itself, or
\item[(ii)]
$\cannbha{\crcomp}$ is contained in a unique canonical neighbourhood $\cannbha{\crcomp'}$ of another critical component $\crcomp'$ of $\partf$ such that $\cannbha{\crcomp'}$ is not a $2$-disk.
\end{enumerate}
\end{lemma}
\begin{proof}
Let $\regnbhAll$ be the union of atoms of all critical components of $\partf$.
Then every connected component $B$ of $\overline{\Mman\setminus\regnbhAll}$ is diffeomorphic to the cylinder $S^1\times[0,1]$ and the restriction $\func|_{B}$ has no critical points.

Notice that $\overline{\Mman\setminus\cannbha{\crcomp}}$ is connected since $\cannbha{\crcomp}$ is a $2$-disk.
Also, there exists a unique connected component $B$ (being a cylinder $S^1\times[0,1]$) of $\overline{\Mman\setminus\regnbhAll}$ such that $\partial\cannbha{\crcomp}\subset B$.
Then $\cannbha{\crcomp}\cup B$ is also a $2$-disk.

Let $n$ be the total number of critical components of $\partf$ in $\overline{\Mman\setminus\cannbha{\crcomp}}$.

If $n=0$, then $\cannbha{\crcomp}\cup B=\Mman$. Whence $\Mman$ is a $2$-disk.

Suppose that $n\geq 1$. 
Let $\gamma$ be another connected component of $\partial B$ distinct from $\partial\cannbha{\crcomp}$.
Then there exists an atom $\regnbha{\crcomp'}$ of some critical component $\crcomp'$ of $\partf$ such that $\gamma\subset\partial\regnbha{\crcomp'}$.
Since $\cannbha{\crcomp}\cup B$ is a $2$-disk, we see that it is contained in $\cannbha{\crcomp'}$.
If $\cannbha{\crcomp'}$ is not a $2$-disk, then the lemma is proved.
Otherwise, the number of critical components in $\overline{\Mman\setminus\cannbha{\crcomp'}}$ is less than in $\overline{\Mman\setminus\cannbha{\crcomp}}$ and the lemma holds by the induction on $n$.
\end{proof}

\begin{example}\rm
Let $\torus$ be a $2$-torus embedded in $\RRR^3$ as shown in Figure~\ref{fig:cannbh} and $\func:\torus\to\RRR$ be the projection onto the vertical line.
Figure~\ref{fig:cannbh}a) shows the critical components of level-sets of $\func$, and Figure~\ref{fig:cannbh}b) presents blackened canonical neighbourhoods of three critical components of $\partf$ containing canonical neighbourhoods of all other critical components of $\partf$.
\end{example}
\begin{figure}[ht]
\begin{tabular}{ccc}
\includegraphics[height=3cm]{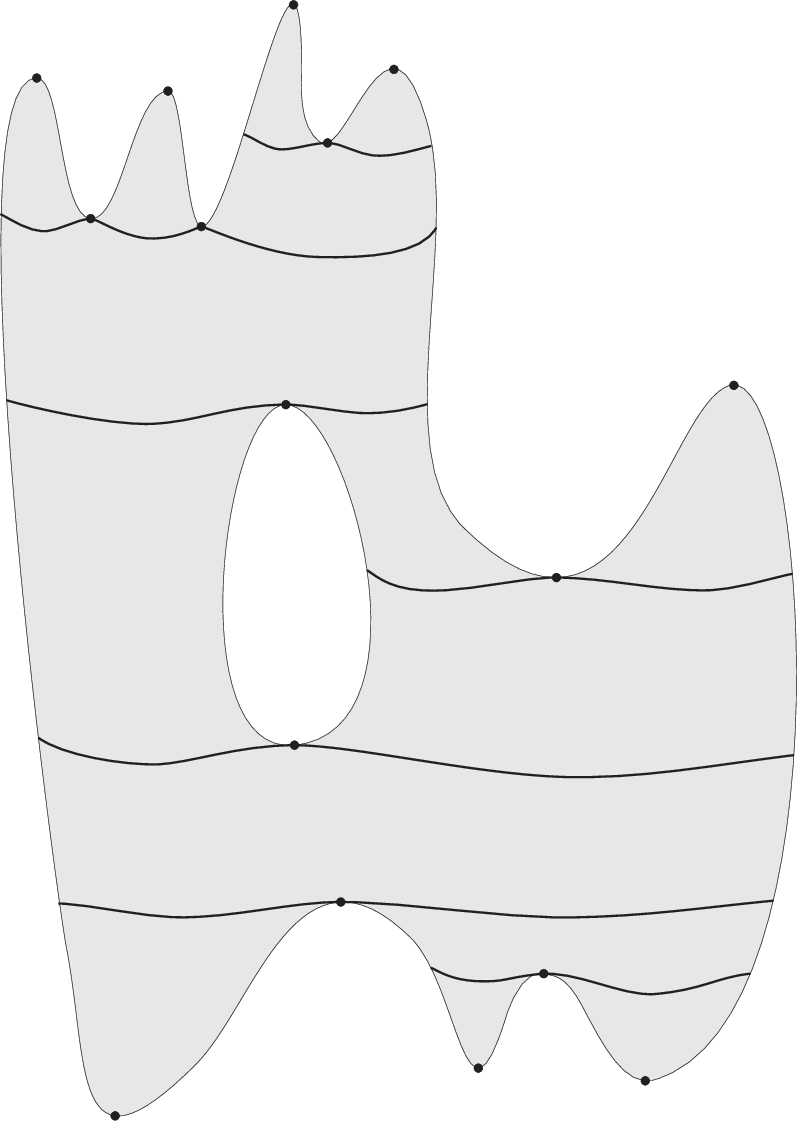}
& \qquad\qquad\qquad &
\includegraphics[height=3cm]{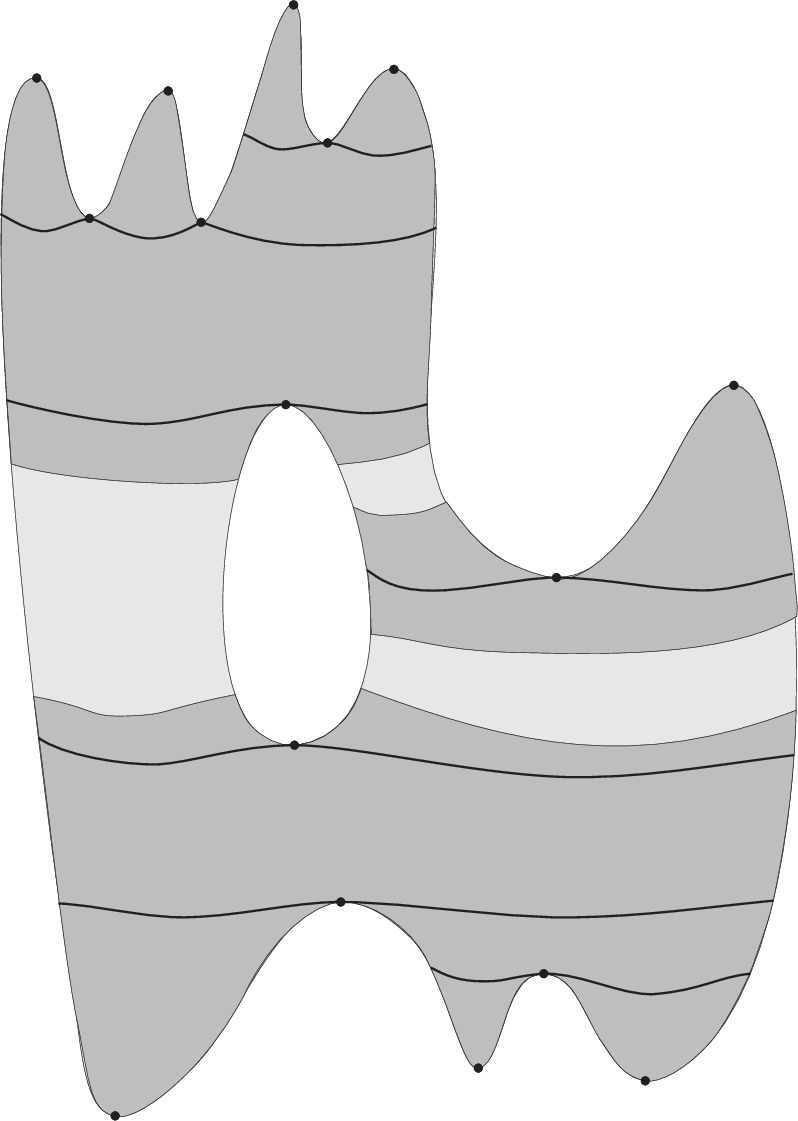}
\\
a) & & b) 
\end{tabular}
\caption{}\protect\label{fig:cannbh}
\end{figure}

\subsection{Canonical neighbourhoods of negative Euler characteristic.}
Suppose $\Mman$ is not a $2$-disk.
Let $\crcomp_1,\ldots,\crcomp_r$ be all the critical components of $\partf$ whose canonical neighbourhoods are not $2$-disks.
By Lemma\;\ref{lm:crcomp-2disk} this collection is non-empty and by Lemma\;\ref{lm:properties_of_atoms} $\cannbha{\crcomp_i}\cap\cannbha{\crcomp_j}=\varnothing$ for $i\not=j$.
Moreover, again by Lemma\;\ref{lm:crcomp-2disk}, any other critical component of $\partf$ is contained in some $\cannbha{\crcomp_i}$.
It follows that $\overline{\Mman\setminus\cup_{i=1}^{r}\cannbha{\crcomp_i}}$ contains no critical points of $\func$, whence it is a disjoint union of cylinders $S^1\times I$.
Therefore 
\begin{equation}\label{equ:chiM_sum_chiNKi}
\chi(\Mman)=\sum_{i=1}^{r}\chi(\cannbha{\crcomp_i}).
\end{equation}

The following two statements will be used for the construction of a surface $\Mmanneg$ of Theorem\;\ref{th:stab1-repr}, see\;\S\ref{sect:proof_th:stab1-repr}.
\begin{lemma}\label{lm:cannhb_chi_neq}
The following conditions are equivalent:
\begin{enumerate}
 \item[\rm(1)]
$\chi(\Mman)<0$;
 \item[\rm(2)]
$\chi(\cannbha{\crcomp_i})<0$ for some $i=1,\ldots,r$.
\end{enumerate}
\end{lemma}
\begin{proof}
(1)$\Rightarrow$(2).
As $\chi(\Mman)<0$, we get from\;\eqref{equ:chiM_sum_chiNKi} that $\chi(\cannbha{\crcomp_i})<0$ for some $i$.

The implication (2)$\Rightarrow$(1) follows from Corollary\;\ref{cor:incompr_chiM_chiN}.
\end{proof}

\begin{corollary}\label{cor:concomp-M-Rl0}
Let $\crcomp_1,\ldots,\crcomp_k$ be all the critical components of $\partf$ whose canonical neighbourhoods have negative Euler characteristic and $\regnbha{\crcomp_1},\ldots,\regnbha{\crcomp_k}$ be their atoms.
Put $\regnbhNDn:=\cup_{i=1}^{k}\regnbha{\crcomp_i}$.
If $\regnbhNDn\not=\varnothing$, then every connected component $B$ of $\overline{\Mman\setminus\regnbhNDn}$ is either a $2$-disk, or a cylinder, or a M\"obius band.
\end{corollary}
\begin{proof}
Since the homomorphism $\pi_1\regnbha{\crcomp_i}\to\pi_1\Mman$ is non-trivial for each $i$, it follows from Corollary\;\ref{cor:incompr_complement} that $B$ is incompressible.
Suppose $\chi(B)<0$.
Notice that $\func$ takes constant values of $\partial B$.
Then by Lemma\;\ref{lm:cannhb_chi_neq} there exists a critical component $\crcomp \subset B$ of $\partf$ such that the canonical neighbourhood $\cannbh$ of $\crcomp$ with respect to $\func|_{B}$ has negative Euler characteristic.
It follows that the homomorphisms $\pi_1\cannbh \to \pi_1 B \to \pi_1 \Mman$ induced by the inclusions $\cannbh\subset B \subset \Mman$ are monomorphisms, so $\cannbh$ is incompressible in $\Mman$.
This implies that $\cannbh$ is a canonical neighbourhood of $\crcomp$ with respect to $\func$.
But since $\chi(\cannbh)<0$, we should have that $\cannbh\subset\regnbhNDn$, which contradicts to the assumption.
\end{proof}

\section{Deformations of incompressible subsurfaces}\label{sect:deform-incompr-surf}
The aim of this section is to extend some results of~\cite{JacoShalen:Topology:1977} concerning incompressible subsurfaces, see Proposition\;\ref{pr:incompr-surf-homotopy}.
\subsection{$\pm$-twist.}
Let $\gamma\subset\Int\Mman$ be a two-sided simple closed curve, $\Uman$ be its regular neighbourhood diffeomorphic to $S^1\times[-1,1]$ so that $\gamma$ correspond to $S^1\times 0$.
Take a function $\mu:[-1,1]\to[0,1]$ such that $\mu=0$ near $\{\pm1\}$ and $\mu=1$ on some neighbourhood of $0$.
Define the following homeomorphism $g_{\gamma}:\Mman\to\Mman$ by
\begin{equation}\label{equ:plus-minus-twist}
g_{\gamma}(x)=
\begin{cases}
(z\, e^{2\pi i \mu(t)}, t), & x=(z,t)\in S^1\times[-1,1]\cong\Uman \\
x, & x\in\Mman\setminus\Uman,
\end{cases}
\end{equation}
see Figure\;\ref{fig:pmtwist}.
Then $g_{\gamma}$ is fixed on some neighbourhood of $\overline{\Mman\setminus\Uman}$ and isotopic to $\id_{\Mman}$ via an isotopy supported in $\Int\Uman$.
Evidently, $g_{\gamma}$ is a product of Dehn twists in opposite directions along the curves parallel to $\gamma$.
Therefore we will call $g_{\gamma}$ a \myemph{$\pm$-twist near $\gamma$}.
\begin{figure}[ht]
\includegraphics[height=2cm]{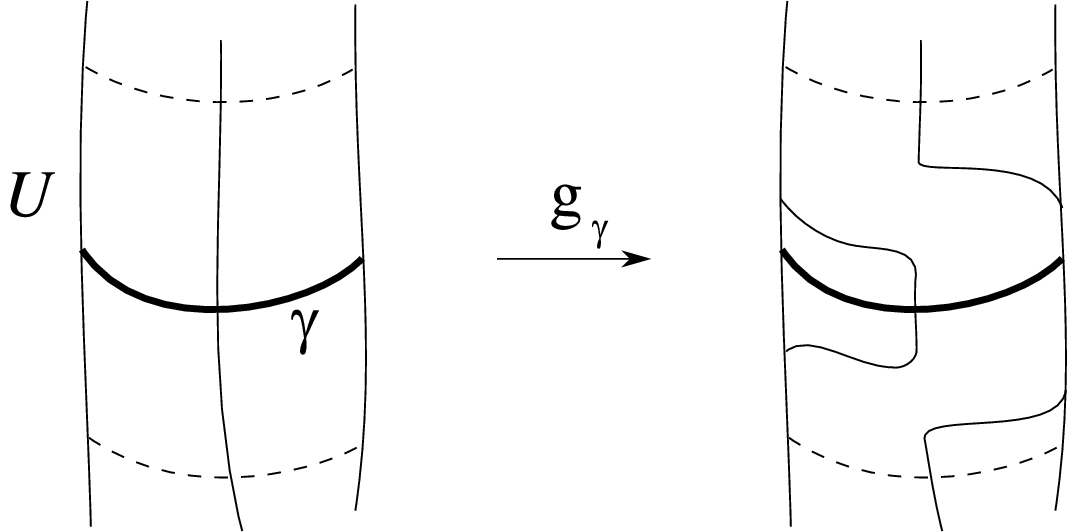}
\caption{$\pm$-twist}\protect\label{fig:pmtwist}
\end{figure}

The following lemma is a particular case of \cite[Lm.~6.1]{Epstein:AM:1966}.
\begin{lemma}\label{lm:invar_curves_under_homotopy}{\rm\cite[Lm.~6.1]{Epstein:AM:1966}.}
Suppose $\chi(\Mman)<0$.
Let $\gamma\subset\Int\Mman$ be a simple closed curve which does not bound a $2$-disk nor a M\"obius band, $\dif:\Mman\to\Mman$ be a homeomorphism homotopic to $\id_{\Mman}$ and such that $\dif(\gamma)=\gamma$.
Let also $\Hdif:\Mman\times I\to\Mman$ be any homotopy of $\id_{\Mman}$ to $\dif$.
Then there exists another homotopy $\Gdif_t:\Mman\times I\to\Mman$ of $\id_{\Mman}$ to $\dif$ such that $\Gdif_t(\gamma)=\gamma$ and $\Gdif_t=\Hdif_t$ on $\overline{\Mman\setminus\Uman}$ for all $t\in I$.

Moreover, there exists $m\in\ZZZ$ and a homotopy $\Gdif':\Mman\times I\to\Mman$ of $\id_{\Mman}$ to $g_{\gamma}^{m}\circ\dif$ such that $\Gdif'_t=\Gdif$ outside $\Uman$ and $\Gdif'_t$ is fixed on $\gamma$ for all $t\in I$.
\end{lemma}

The following statement is also well-known.

\begin{lemma}\label{lm:dehn-twists-along-boundary}
Let $\Mman$ be a surface with $\chi(\Mman)<0$.
Suppose $\partial\Mman\not=\varnothing$ and let $\gamma_1,\ldots,\gamma_l$ be all the connected components of $\partial\Mman$.
For each $i=1,\ldots,l$ let $\tau_i$ be a Dehn twist along the curve parallel to $\gamma_i$ and fixed on $\partial\Mman$.
Let $m_1,\ldots,m_l\in\ZZZ$ be integer numbers not of all are equal to zero.
Then the homeomorphism $\tau_1^{m_1}\circ\cdots\circ\tau_l^{m_l}$ is \myemph{not homotopic to $\id_{\Mman}$} via a homotopy fixed on $\partial\Mman$.
\end{lemma}

\subsection{Deformations of incompressible subsurfaces.}
Let $\Mman$ be a surface distinct from the $2$-sphere $S^2$ and the projective plane $\prjplane$, $\nmanif\subset\Mman$ be an incompressible subsurface, and $\nmanif_1,\ldots,\nmanif_k$ be all of its connected components.
Let also $\dif:\Mman\to\Mman$ be a homeomorphism homotopic to $\id_{\Mman}$ and $\Hdif:\Mman\times I\to\Mman$ be any homotopy of $\id_{\Mman}$ to $\dif$.

The following Proposition\;\ref{pr:incompr-surf-homotopy} follows the line of\;\cite[Lm.\;4.2]{JacoShalen:Topology:1977}.
In fact the first part of statement (B) is a particular case of that lemma.

\begin{proposition}\label{pr:incompr-surf-homotopy}{\rm c.f.\;\cite[Lm.\;4.2]{JacoShalen:Topology:1977}}
{\rm(A)}~If $\nmanif_j$ is not a cylinder for some $j$, then $\dif(\nmanif_j)\cap\nmanif_j\not=\varnothing$.

{\rm(B)}~Suppose $\chi(\nmanif_j)<0$ and $\dif(\nmanif_j)\subset\nmanif_j$ for some $j$.
Then there exists a homotopy $\Gdif:\nmanif_j\times I\to\nmanif_j$ of the identity map $\id_{\nmanif_j}$ to the restriction $\dif|_{\nmanif_j}$ such that $\Gdif_t(x)=\Hdif_t(x)$ whenever $\Hdif(x\times I)\subset\nmanif_j$.

Moreover, suppose $\Hdif(\gamma\times I)\subset\gamma$ for each connected component $\gamma$ of $\partial\nmanif_j$.
Extend $\Gdif$ to a map $\Gdif:\Mman\times I\to\Mman$ by $\Gdif_t=\Hdif_t$ on $\Mman\setminus\nmanif_j$.
Then $\Gdif$ is a homotopy of $\id_{\Mman}$ to $\dif$.

{\rm(C)}~Suppose $\chi(\nmanif_j)<0$ and $\dif(\nmanif_j)=\nmanif_j$ for all $j=1,\ldots,k$.
Then there exists a homotopy $\Gdif:\Mman\times I\to\Mman$ of $\id_{\Mman}$ to $\dif$ such that $\Gdif(\nmanif_j\times I)\subset\nmanif_j$ for all $j=1,\ldots,k$ and $\Gdif(B\times I)\subset B$ for every connected component $B$ of $\overline{\Mman\setminus\nmanif}$.

{\rm(D)}~
Suppose $\chi(\nmanif_j)<0$ and $\dif$ is fixed on $\nmanif$ for all $j=1,\ldots,k$.
Then there exists a homotopy of $\id_{\Mman}$ to $\dif$ fixed on $\nmanif$.
\end{proposition}
\begin{proof}
First we make the following remark which repeats the key arguments of\;\cite[Lm.\;4.2]{JacoShalen:Topology:1977}.
For $j=1,\ldots,k$ let $p_j:\tMman_j\to\Mman$ be the covering map corresponding to the subgroup $\pi_1\nmanif_j$ of $\pi_1\Mman$.
Then the embedding $i:\nmanif_j\subset\Mman$ lifts to the embedding $i^{*}:\nmanif_j\to\tMman_j$ which induces an isomorphism between $\pi_1\nmanif_j$ and $\pi_1\tMman_j$.
Denote $\tnmanif_j=i^{*}(\nmanif_j)$.
Then we have the following commutative diagram:
$$
\xymatrix{
\tnmanif_j\; \ar@{^{(}->}[rr] & & \tMman_j \ar[d]^{p_j}\\
\nmanif_j\; \ar[u]^{\cong} \ar[urr]^{i^{*}} \ar[rr]^{i}  & & \Mman
}
$$
Since $\tMman_j$ and $\nmanif_j$ are aspherical, it follows from Whitehead's theorem that $\tnmanif_j$ is a strong deformation retract of $\tMman_j$.
Then every connected component of $\Int(\tMman_j\setminus\tnmanif_j)$ is an open cylinder.
Let $\Hdif:\nmanif_j\times I\to\Mman$ be any homotopy between the identity embedding $\Hdif_0=i:\nmanif_j\subset\Mman$ and $\Hdif_1=\dif|_{\nmanif_j}$.
Then there exists a lifting $\tHdif:\nmanif_j\times I\to\tMman_j$ such that $\tHdif_0=i^{*}$ and $p_j\circ\tHdif=\Hdif$.
Denote $\tnmanif'_j=\tHdif_1(\nmanif_j)$.
Since both $\tnmanif_j$ and $\tnmanif'_j$ are deformational retracts of $\tMman_j$, they are incompressible in $\tMman_j$.

\medskip 

(A) Suppose $\dif(\nmanif_j)\cap\nmanif_j=\varnothing$.
Then $\Int(\tnmanif'_j)$ is included into some connected component $C$ of $\Int(\tMman_j\setminus\tnmanif_j)$ being a cylinder.
Since $\tnmanif'_j$ is incompressible in $\Mman$, it is also incompressible in $C$, whence $\tnmanif'_j$ and therefore $\nmanif_j$ are cylinders.
Thus if $\nmanif_j$ is not a cylinder, then we obtain that $\dif(\nmanif_j)\cap\nmanif_j\not=\varnothing$.

\medskip 

(B) Let $r_j:\tMman_j\to\tnmanif_j$ be any retraction.
Then the map $$\Gdif = p_j \circ r_j \circ \tHdif \ : \ \nmanif_j\times I\;\to\;\nmanif_j$$ is a homotopy of $\id_{\nmanif_j}$ to $\dif|_{\nmanif_j}$ in $\nmanif_j$.
It is easy to see that $\Gdif_t(x)=\Hdif_t(x)$ whenever $\Hdif(x\times I)\subset\nmanif_j$.

Suppose that $\Hdif(\gamma\times I)\subset\gamma\subset\nmanif_j$ for each connected component $\gamma$ of $\partial\nmanif_j$.
Then by the construction $\Gdif_t=\Hdif_t$ on $\partial\nmanif_j$.
Notice that $\partial\nmanif_j$ separates $\Mman$.
Extend $\Gdif$ to all of $\Mman\times I$ by $\Gdif=\Hdif$ of $(\Mman\setminus\nmanif_j)\times I$.
Then $\Gdif$ is continuous, $\Gdif_0=\id_{\Mman}$ and $\Gdif_1=\dif$.

\medskip 

(C)
Suppose $\chi(\nmanif_j)<0$ and $\dif(\nmanif_j)=\nmanif_j$ for all $j=1,\ldots,k$.
Let $\gamma_1,\ldots,\gamma_l$ be all the connected components of $\partial\nmanif$.
Since $\nmanif$ is incompressible, we have by Corollary\;\ref{cor:incompr_chiM_chiN} that $\chi(\Mman)\leq\chi(\nmanif_j)<0$ as well.
Moreover, by (B) for each $j$ the restriction $\dif|_{\nmanif_j}$ is a homeomorphism of $\nmanif_j$ homotopic in $\nmanif_j$ to $\id_{\nmanif_j}$.
This, in particular, implies that $\dif(\gamma_i)=\gamma_i$ for $i=1,\ldots,l$.

Then by Lemma\;\ref{lm:invar_curves_under_homotopy} we can suppose that $\Hdif(\gamma_i\times I)\subset\gamma_i$ for all $i=1,\ldots,l$ as well.
Moreover, due to (B) it can be additionally assumed that $\Hdif(\nmanif_j\times I)\subset \nmanif_j$.

Let $B$ be a connected component of $\overline{\Mman\setminus\nmanif}$.
Since $\nmanif$ is incompressible, $B$ is not a $2$-disk.
Then by Corollary\;\ref{cor:incompr_complement} $B$ is incompressible.
Therefore we can apply statement (B) to $B$ and change the homotopy $\Gdif$ on $B\times I$ so that $\Gdif(B\times I)\subset B$.

\medskip 

(D) Suppose $\dif$ is fixed on $\nmanif$.
For each $i$ let $\Uman_i$ be a regular neighbourhood of $\gamma_i$, and $g_i$ be a $\pm$-twist near $\gamma_i$ supported in $\Uman_i$.
We can assume that $\Uman_i\cap\Uman_j=\varnothing$ for $i\not=j$.
Then by Lemma\;\ref{lm:invar_curves_under_homotopy} there exist integer numbers $m_1,\ldots,m_l\in\ZZZ$ and a homotopy $\Gdif:\Mman\times I\to\Mman$ of $\id_{\Mman}$ to a homeomorphism $\dif':=g_1^{m_1}\circ\cdots\circ g_{l}^{m_l}\circ\dif$ such that $\Gdif_t$ is fixed on $L$ for each $t\in I$.
By (C) we can also assume that $\Gdif(\nmanif_j\times I)\subset\nmanif_j$ and $\Gdif(B\times I)\subset B$ for every connected component of $\overline{\Mman\setminus\nmanif}$ and each $j=1,\ldots,k$.

In particular, we see that the restriction $\dif'|_{\nmanif}$ is homotopic to $\id_{\nmanif}$ relatively $\partial\nmanif$.
But this restriction is evidently a product of Dehn twists along boundary components of $\nmanif$.
Since $\chi(\nmanif_j)<0$ for all $j$, we get from Lemma\;\ref{lm:dehn-twists-along-boundary} that $m_i=0$ for all $i=1,\ldots,l$.
Hence $\dif'=\dif$.
Thus $\Gdif$ is in fact a homotopy between $\id_{\Mman}$ and $\dif$ relatively $\partial\nmanif$.
Since $\partial\nmanif$ separates $\Mman$, and $\id_{\Mman}$ and $\dif$ are fixed on $\nmanif$, we can change $\Gdif$ on $\nmanif\times I$ by $\Gdif_t(x)=x$.
This gives a homotopy between $\id_{\Mman}$ and $\dif$ relatively to $\nmanif$.
\end{proof}

\section{Automorphisms of cellular subdivisions}\label{sect:cell-auto}
Let $\nmanif$ be a compact surface and $\cwpart=\{e_{\lambda}\}_{\lambda\in\Lambda}$ be some partition of $\nmanif$ into a disjoint family of connected orientable submanifolds.
Say that a homeomorphism $\dif:\nmanif\to\nmanif$ is a \myemph{$\cwpart$-homeomorphism} provided it yields a permutation of elements of $\cwpart$, that is for each $e\in\cwpart$ its image $\dif(e)$ also belongs to $\cwpart$.
An element $e\in\cwpart$ will be called \myemph{$\dif$-invariant\/} if $\dif(e)=e$.
Say that $e$ is \myemph{$\dif^{+}$-invariant} (\myemph{$\dif^{-}$-invariant}) if the restriction $\dif|_{e}:e \to e$ is a preserving (reversing) orientation map.
We will also say that $\dif$ is \myemph{$\cwpart$-trivial} if each $e\in\cwpart$ is $\dif^{+}$-invariant.

\begin{remark}\label{rem:no-h-minus-0-cells}\rm
Notice that we can say that a map $\dif:e \to e$ preserves or reverses orientation only if $\dim e \geq 1$.
To each $0$-dimensional element $e\in\cwpart$ (being of course a point) we formally assign a ``positive orientation'' and assume that \myemph{by definition every cellular homeomorphism preserves orientation of each invariant $0$-element of $\cwpart$.}
\end{remark}

\begin{example}\rm
Let $\Mman$ be a connected surface and $\kgraph\subset\Int\Mman$ be an embedded finite connected graph.
Assume that $\kgraph$ is a subcomplex of $\Mman$ with respect to some triangulation of $\Mman$.
By $\regnbhg$ we will denote a regular neighbourhood of $\kgraph$.
Following~\cite{JacoShalen:Topology:1977} define a \myemph{canonical\/} neighbourhood $\cannbhg$ of $\kgraph$ to be the union of a regular neighbourhood $\regnbhg$ of $\kgraph$ with those connected components of $\Mman\setminus\regnbhg$ which are $2$-disks.
Notice that $\cannbhg\setminus\kgraph$ is a disjoint union of open $2$-disks and half-open cylinders $S^1\times(0,1]$ with $S^1\times\{1\}$ corresponding boundary components of $\partial\cannbhg$.
Thus we obtain a natural partition of $\cannbhg$ by vertexes and edges of $\crcomp$ and connected components of $\cannbhg\setminus\kgraph$.
We shall denote this partition by $\cwpartg$.
\end{example}

Now let $\cwpart$ be a cellular subdivision of $\nmanif$.
Denote by $\nmanif_i$ $(i=0,1,2)$ the $i$-th skeleton of $\nmanif$.
In particular, $\nmanif_1$ is a finite connected subgraph in $\nmanif$ such that $\nmanif\setminus\nmanif_1$ is a disjoint union of $2$-disks.
Let $c_i$ $(i=0,1,2)$ be the total number of $i$-cells of $\Delta$.
Then of course $\chi(\nmanif) = c_0-c_1+c_2.$

Let $\chains=\{\chains_i,\partial_i\}$ be the $\RRR$-chain complex of $\nmanif$ corresponding to a given cellular subdivision.
Thus $\chains_i$ is a real vector space of dimension $c_i$ with the \myemph{oriented} $i$-cells of $\cwpart$ as a basis.
Then every $\cwpart$-homeo\-morhism $\dif$ induces a chain automorphism
$\{\dif_i:\chains_i\to \chains_i, \ i=0,1,2\}$ of $\chains$.

Recall that for each continuous mapping $\dif:\nmanif\to\nmanif$ we can define its \myemph{Lefschetz number\/} $L(\dif)$ by the formula:
$$
L(\dif) = \tr(\bdifM_0) - \tr(\bdifM_1) + \tr(\bdifM_2),
$$
where $\bdifM_i:H_i(\nmanif,\RRR) \to H_i(\nmanif,\RRR)$ is the induced homomorphism of the corresponding homology groups and $\tr$ is the trace of this homomorphism.
If $\dif$ is cellular, then $L(\dif)$ can also be calculated via the chain homomorphisms $\dif_i$ by:
$$
L(\dif) = \tr(\dif_0) - \tr(\dif_1) + \tr(\dif_2).
$$

The following theorem is relevant to\;\cite[Lm.\;4.4]{JacoShalen:Topology:1977} being a statement about periodic homeomorphisms.
\begin{theorem}\label{th:aut-surf-graph}{\rm c.f.\;\cite[Lm.\;4.4]{JacoShalen:Topology:1977}.}
Let $\Mman$ be a compact surface, $\kgraph\subset\Mman$ a connected subgraph, $\cannbhg$ be a canonical neighbourhood of $\kgraph$.
Let also $\dif:\Mman\to\Mman$ a homeomorphism such that $\dif$ is homotopic to $\id_{\Mman}$, $\dif(\kgraph)=\kgraph$, and $\dif$ preserves the set of vertexes of $\kgraph$ of degree $2$, and $\dif(\cannbhg)=\cannbhg$.
In particular, $\dif|_{\cannbhg}$ is a $\cwpartg$-homeomorphism.
\begin{enumerate}
\item 
If $\chi(\cannbhg)<0$, then $\dif$ is $\cwpartg$-trivial.
\item 
Suppose that $\cannbhg=\Mman$, $\Mman$ is orientable, and $\chi(\Mman)\geq 0$.
Then every \myemph{annulus\/} $a\in\cwpartg$ is $\dif^{+}$-invariant, and the total number of $\dif$-invariant \myemph{cells\/} of $\cwpartg$ is equal to $\chi(\Mman)$.
\end{enumerate}
\end{theorem}
The proof of Theorem~\ref{th:aut-surf-graph} will be given in\;\S\ref{sect:proof-th:aut-surf-graph}.
It is based on Proposition~\ref{pr:incompr-surf-homotopy} and on the following statement.

\begin{proposition}\label{pr:k_Lh}
Let $\nmanif$ be a {\bfseries closed, orientable\/} surface endowed with some cellular subdivision $\cwpart$ and $\dif:\nmanif\to\nmanif$ be a $\cwpart$-homeomorphism \myemph{\/preserving orientation\/} of $\nmanif$ and being not $\cwpart$-trivial, i.e. $\dif(e)\not=e$ for some cell $e\in\cwpart$. 
Then the number of $\dif$-invariant cells of $\cwpart$ is precisely equal to $L(\dif)$.
In particular, $L(\dif)\geq0$.
\end{proposition}
\begin{proof}
Let $\fc_{i}$ $(i=0,1,2)$ be the number of $\dif$-invariant $i$-cells of $\cwpart$ and $\fc:=\fc_0+\fc_1+\fc_2$. 
We will show that 
\begin{equation}\label{equ:fc_Lh}
\fc_{i} = (-1)^{i} \; \tr(\dif_{i}),
\end{equation}
which will imply
$$
\fc = \sum_{i=0}^{2} \fc_i =  \sum_{i=0}^{2} (-1)^{i} \; \tr(\dif_{i}) = L(\dif).
$$

To prove~\eqref{equ:fc_Lh} we have to show that there are no $\dif^{-}$-invariant $0$- and $2$-cells and no $\dif^{+}$-invariant $1$-cells.
For $0$-cells this holds by Remark~\ref{rem:no-h-minus-0-cells} and for $2$-cells from the assumptions that $\nmanif$ is orientable and $\dif$ preserves orientation.

Let $e$ be an $\dif$-invariant $1$-cell and $f_0$ and $f_1$ be two $2$-cells that are incident to $e$. It is possible of course that $f_0=f_1$.
Since $\dif$ preserves orientation, it follows that

(a) either $\dif_2(f_j) = +f_j$ for $j=0,1$, and $\dif_1(e)=+e$,

(b) or $\dif_2(f_j) = +f_{1-j}$ for $j=0,1$, and $\dif_1(e)=-e$.

The following Claim~\ref{clm:h_fixes_nbh_e} implies that in the case (a) $\dif$ is $\cwpart$-trivial.
Since $\dif$ is not $\cwpart$-trivial, we will get from (b) that all $\dif$-invariant $1$-cells are $\dif^{-}$-invariant.

\begin{claim}\label{clm:h_fixes_nbh_e}
Suppose that there exists a $1$-cell $e \in \cwpart$ such that
\begin{enumerate}
\item[\rm(i)]
$\dif_1(e)=+e \in \chains_1$ and
\item[\rm(ii)]
$\dif$ preserves each $2$-cell which is adjacent to $e$.
\end{enumerate}
Then $\dif$ is $\cwpart$-trivial.
\end{claim}
\begin{proof}
Notice that for each vertex $v\in\nmanif_0$ the inclusion $\nmanif_1\subset\nmanif$ induces a cyclic ordering of edges that are incident to $v$. 

Let $v$ be a vertex of $e$. 
Then it follows from (i) and (ii) that all of the $1$- and $2$-cells incident to $v$ are $\dif^{+}$-invariant.
Moreover, for each $1$-cell that is incident to $v$ the conditions (i) and (ii) also hold true.
Since $\nmanif$ is connected, it follows that $\dif$ is $\cwpart$-trivial.
\end{proof}
Proposition~\ref{pr:k_Lh} is completed.
\end{proof}

\begin{corollary}\label{cor:chi-neg-h-triv}
Let $\nmanif$ be a {\bfseries closed} surface, $\cwpart$ be a cellular subdivision of $\Mman$, and $\dif:\nmanif\to\nmanif$ be a $\cwpart$-homeomorphism.
If $\dif$ is isotopic to $\id_{\nmanif}$, then each of the following conditions implies that $\dif$ is $\cwpart$-trivial:
\begin{enumerate}
\item 
$\chi(\nmanif)<0$;

\item
$\chi(\nmanif)\geq0$ and the total number of $\dif^{+}$-invariant $2$-cells is greater than $\chi(\nmanif)$.
\end{enumerate}
\end{corollary}
\begin{proof}
Since $\dif$ is isotopic to $\id_{\nmanif}$, we have that $ L(\dif) = L(\id_{\nmanif}) = \chi(\nmanif)$.

If $\nmanif$ is orientable, then $\dif$ preserves orientation and by Proposition~\ref{pr:k_Lh} $\dif$ is either $\cwpart$-trivial or has exactly $\chi(\nmanif)\geq0$ invariant cells.
Each of the conditions (1) and (2) implies that the number of $\dif$-invariant cells is not equal to $\chi(\nmanif)$.
Hence $\dif$ is $\cwpart$-trivial.

\medskip

Suppose that $\nmanif$ is non-orientable and let $p:\tnmanif\to\nmanif$ be its oriented double covering.
Then $\cwpart$ lifts to some cellular subdivision $\tcwpart$ of $\tnmanif$ and $\dif$ lifts to a unique 
$\tcwpart$-cellular homeomorphism $\tdifM$ of $\tnmanif$ which is isotopic to $\id_{\tnmanif}$.
Therefore $ L(\tdifM) = L(\id_{\tnmanif}) = \chi(\tnmanif) = 2\chi(\nmanif)$.

We claim that every of the conditions (1) and (2) implies that $\tdifM$ is $\tcwpart$-trivial, whence $\dif$ will be $\cwpart$-trivial.

(1) If $\chi(\nmanif)<0$, then $\chi(\tnmanif)<0$, whence $\tdifM$ is $\tcwpart$-trivial.

(2) Suppose that $\chi(\nmanif)\geq0$ and the total number $b$ of $\dif^{+}$-invariant $2$-cells is greater than $\chi(\nmanif)$.
Let $e$ be an $\dif^{+}$-invariant $2$-cell of $\cwpart$ and $\tilde{e}_1$ and $\tilde{e}_2$ be its liftings in $\tcwpart$.
Then they are $\tdifM^{+}$-invariant.
Hence $\tdifM$ has at least $2b>2\chi(\nmanif)=\chi(\tnmanif)$ invariant cells.
Then by Proposition~\ref{pr:k_Lh} $\tdifM$ is $\tcwpart$-trivial.
\end{proof}

\subsection{Proof of Theorem~\ref{th:aut-surf-graph}}\label{sect:proof-th:aut-surf-graph}
Let $\dif:\Mman\to\Mman$ be a homeomorphism homotopic to the identity and such that $\dif|_{\cannbhg}$ is a $\cwpartg$-homeomorphism.
Let $\gamma_1,\ldots,\gamma_b$ be all the connected components of $\partial\cannbhg$, and $a_1,\ldots,a_b$ be the annuli of $\cwpartg$ corresponding to them, so that $\gamma_i\subset a_i$.
Shrink every $\gamma_i$ to a point $x_i$ and denote the obtained surface by $\hcannbhg$.
Then $\hcannbhg$ is a closed orientable surface and $\cwpartg$ yields an evident cellular partition $\tcwpart$ of $\hcannbhg$ such that each annulus $a_i$ corresponds to a certain $2$-cell $\widehat{a}_i\in\tcwpart$.

Also notice that $\chi(\hcannbhg) = \chi(\cannbhg)+b$.

\begin{claim}\label{clm:1_2_imply_hNk_idNk}
Suppose that either $\chi(\cannbhg)<0$ or $\cannbhg=\Mman$.
Then 
\begin{enumerate}
\item[(a)]
$\dif|_{\cannbhg}$ is homotopic to $\id_{\cannbhg}$ in $\cannbhg$.
\item[(b)]
$\dif(\gamma_i)=\gamma_i$ for $i=1,\ldots,b$ and $\dif$ preserves orientation of $\gamma_i$;
\item[(c)]
$\dif$ induces some $\tcwpart$-homeomorphism $\hdifM:\hcannbhg\to\hcannbhg$ homotopic to $\id_{\hcannbhg}$ with respect to $\{x_1,\ldots,x_b\}$, in particular, every $2$-cell $a_i\in\tcwpart$ is $\hdifM^{+}$-invariant;
\item[(d)]
$L(\hdifM)=L(\id_{\hcannbhg})=\chi(\hcannbhg)= \chi(\cannbhg)+b.$
\end{enumerate}
\end{claim}
\begin{proof}
(a) For $\cannbhg=\Mman$ this statement is trivial.
If $\chi(\cannbhg)<0$, then by (B) of Proposition~\ref{pr:incompr-surf-homotopy} (or directly by\;\cite[Lm.\;4.1]{JacoShalen:Topology:1977}) $\dif|_{\cannbhg}$ is homotopic to $\id_{\cannbhg}$ in $\cannbhg$.
All other statements (b)-(d) follow from (a).
\end{proof}

Now we can complete Theorem~\ref{th:aut-surf-graph}.

(1) Suppose that $\chi(\cannbhg)<0$.
If also $\chi(\hcannbhg)<0$, then by (1) of Corollary~\ref{cor:chi-neg-h-triv} $\hdifM$ is $\tcwpart$-trivial, whence $\dif$ is $\cwpartg$-trivial as well.

Let $\chi(\hcannbhg)\geq0$.
By Claim\;\ref{clm:1_2_imply_hNk_idNk} $\hdifM$ has at least $b$ $\hdifM^{+}$-invariant $2$-cells $a_1,\ldots,a_b$.
Moreover, since $\chi(\hcannbhg)-b=\chi(\cannbhg)<0$, we obtain that $b>\chi(\hcannbhg)$, whence by (2) of Corollary~\ref{cor:chi-neg-h-triv} $\hdifM$ is $\tcwpart$-trivial.
Therefore $\dif$ is $\cwpartg$-trivial.

(2) Suppose that $\cannbhg=\Mman$ and $\Mman$ is orientable.
It follows from (c) of Claim~\ref{clm:1_2_imply_hNk_idNk} and Proposition~\ref{pr:k_Lh} that $\hdifM$ is either $\tcwpart$-trivial or has exactly $\chi(\hcannbhg)$ invariant cells.
Therefore, $\dif$ is either $\cwpartg$-trivial or has exactly $\chi(\hcannbhg)-b=\chi(\cannbhg)=\chi(\Mman)$ invariant cells.
\endproof

\section{Deformations of diffeomorphism near critical components of $\partf$}\label{sect:deform-diff-crit-comp}
The following two propositions will be crucial for the proof of Theorem\;\ref{th:stab1-repr}.
Suppose $\func:\Mman\to\Psp$ satisfies \AxBd, \AxIsol, and \AxAdm.
\begin{proposition}\label{prop:deform-in-Stabf}
Let $\crcomp$ be a critical component of $\partf$ such that every $z\in\crcomp\cap\singf$ is admissible, $\regnbh$ be its atom, and $\Uman$ be any neighbourhood of $\regnbh$.
Let also $\dif\in\Stabf$.
Suppose that $\dif(\omega)=\omega$ for each leaf $\omega$ of $\partf$ contained in $\crcomp$ and that $\dif$ preserves orientation of $\omega$ whenever $\dim\omega=1$.
Then $\dif$ is isotopic in $\Stabf$ to a diffeomorphism $\dif'\in\Stabf$ such that $\dif'=\dif$ on $\Mman\setminus\Uman$, and $\dif'$ is the identity on some neighbourhood of $\regnbh$ in $\Uman$.
\end{proposition}
\begin{proof}
This proposition follows the line of\;\cite[Th.\;6.2]{Maks:AGAG:2006}.
For the convenience of the reader we will recall the key arguments for the case when $\Mman$ is orientable.
A non-orientable case can be deduced from the orientable one similarly to \cite[Th.\;6.2]{Maks:AGAG:2006}.

As $\Mman$ is orientable, it has a symplectic structure.
Let $\HFld$ be the Hamiltonian vector field of $\func$.
Then $\func$ is constant along orbits of $\HFld$, the set of singular points of $\HFld$ coincides with the set of critical points of $\func$, and the foliation by orbits of $\HFld$ coincides with $\partf$.
In particular, $\HFld$ is tangent to $\partial\Mman$ and therefore generates a flow $\HFlow:\Mman\times\RRR\to\Mman$.

We will now change $\HFld$ on neighbourhoods of admissible critical points of $\func$ similarly to\;\cite[Lm.\;5.1]{Maks:AGAG:2006}.
Let $z\in\singf$ be such a point and $\AFld_z$ be a vector field on some neighbourhood $\Uman_z$ of $z$ satisfying assumptions of Definition\;\ref{defn:admissible_cr_pt}.
Then it follows from (i) of Definition\;\ref{defn:admissible_cr_pt} that for every $x\in\Uman_z$ the vectors $\HFld(x)$ and $\AFld_z$ are parallel each other.
Therefore, using partition unity technique and changing (if necessary) the signs of $\AFld_z$, we can change $\HFld$ near each $z\in\regnbh\cap\singf$ and assume that $\HFld=\AFld_z$ on $\Uman_z$.

\begin{claim}
There exists a neighbourhood $\Uman$ of $\regnbh$ and a unique $\Cinf$ function $\sigma:\Uman\to\RRR$ such that $\dif(x)=\HFlow(x,\sigma(x))$ for all $x\in\Uman$.
\end{claim}
\begin{proof}
Let $z\in\crcomp\cap\singf$.
By assumption $\dif$ preserves leaves of $\partf$ (i.e. orbits of $\HFlow$) in $\crcomp$ with their orientations.
Since $\AFld_z=\HFld$ near $z$, it follows from (ii) of Definition\;\ref{defn:admissible_cr_pt} that there exists a neighbourhood $\Vman_z$ of $z$ and a unique $\Cinf$ function $\sigma_z:\Vman_z\to\RRR$ such that $\dif(x)=\HFlow(x,\sigma_z(x))$.
Then the functions $\{\sigma_{z}\}_{z\in\crcomp\cap\singf}$ yield a unique $\Cinf$ function $\sigma$ on the union $\mathop\cup\limits_{z\in\crcomp\cap\singf}\Vman_{z}$.
It remains to note that $\crcomp\setminus\singf$ is a disjoint union of open intervals, whence $\sigma$ uniquely extends to a $\Cinf$ function on $\regnbh$ such that $\dif(x)=\HFlow(x,\sigma(x))$, see\;\cite[Lm.\;6.4]{Maks:AGAG:2006} for details.
\end{proof}

Then the desired isotopy of $\dif$ to $\dif'$ in $\Stabf$ can be constructed similarly to\;\cite[Lm.\;4.14]{Maks:AGAG:2006}.
Take any $\Cinf$ function $\mu:\Mman\to[0,1]$ such that $\mu=0$ on some neighbourhood of $\overline{\Mman\setminus\Uman}$, $\mu=1$ on $\regnbh$, and $\mu$ is constant along orbits of $\AFld$.
Then the function $\nu=\mu\sigma$ is $\Cinf$ and well-defined on all of $\Mman$.
Consider the following homotopy
\begin{equation}\label{equ:isotopy_fix_on_M_U}
\gdif:\Mman\times I\to\Mman,
\qquad
\gdif_t(x) = \AFlow(x,t\nu(x)).
\end{equation}
Then $\gdif_0=\id_{\Mman}$, $\gdif_t$ is fixed on $\overline{\Mman\setminus\Uman}$, and $\gdif_1=\dif$ on $\regnbh$.
Since $\mu$ is constant along orbits of $\AFld$ and $\dif$ is a diffeomorphism, it follows from\;\cite[Lm.\;4.14]{Maks:AGAG:2006} that $\gdif$ is an isotopy.
Hence $\gdif_t^{-1}\circ\dif:\Mman\to\Mman$, $(t\in I)$, is an isotopy in $\Stabf$ supported in $\Uman$ and deforming $\dif$ to a desired diffeomorphism $\dif'=\gdif^{-1}_1\circ\dif$.
\end{proof}

\begin{proposition}\label{prop:deform-homotopies-in-Stabf}
Let $\Mmanneg\subset\Mman$ be a compact subsurface such that $\partial\Mmanneg$ consists of (regular) leaves of $\partf$.
Suppose $\dif\in\StabIdf$ is fixed on some neighbourhood $\Uman$ of $\Mmanneg$.
Then there exists an isotopy of $\dif$ to $\id_{\Mman}$ in $\Stabf$ fixed on some neighbourhood of $\Mmanneg$.
\end{proposition}
\begin{proof}
Again we will consider only the case when $\Mman$ is orientable.
Let $\HFlow:\Mman\times\RRR\to\Mman$ be the flow constructed in the proof of Proposition\;\ref{prop:deform-in-Stabf}.
Since $\dif\in\StabIdf$, there exists an isotopy $\Gdif:\Mman\times I\to\Mman$ of $\id_{\Mman}$ to $\dif$ in $\Stabf$.
Then is it easy to show that each $\Gdif_t$ preserves orbits of $\HFlow$ on some neighbourhood of $\Mmanneg$, see\;\cite[Lm.\;3.4]{Maks:AGAG:2006}.
Now it follows from\;\cite[Th.\;25]{Maks:TA:2003}, see also\;\cite{Maks:ImSh}, that there exists a continuous function $\Lambda:(\Mman\setminus\singf)\times I\to\RRR$ such that $\Lambda_{t}$ is $\Cinf$ for each $t\in I$, $\Lambda_0=0$, and $\Gdif_t(x)=\HFlow(x,\Lambda_t(x))$ for all $x\in\Mman\setminus\singf$.
Let $\mu:\Mman\to[0,1]$ be a $\Cinf$ function constant along orbits of $\HFld$, $\mu=0$ on $\Mmanneg$, and $\mu=1$ on some neighbourhood of $\overline{\Mman\setminus\Uman}$.
Define the following map $a:\Mman\times I\to\Mman$ by
$$
a(x,t)=
\begin{cases}
\HFlow(x,\mu(x)\Lambda(x,t)), & x\in\Uman \\
\Gdif_t(x), & x\in\Mman\setminus\Uman.
\end{cases}
$$
We claim that $a$ is an isotopy between $\id_{\Mman}$ and $\dif$ in $\Stabf$ fixed on some neighbourhood of $\Mmanneg$.

Since $\mu=1$ on some neighbourhood of $\overline{\Mman\setminus\Uman}$, we see that $a$ is continuous and $a_t$ is $\Cinf$ for each $t$.
Moreover,
$$
a(x,0)=
\begin{cases}
\HFlow(x,0)=x, & x\in\Uman \\
\Gdif_0(x)=x, & x\in\Mman\setminus\Uman.
\end{cases}
$$
Since $\dif$ is fixed on $\Uman$, it follows that $\Lambda(x,1)=0$ on $\Uman$.
Therefore $\mu\Lambda_1=\Lambda_1$ and $a_1=\dif$.
As $\mu=0$ on $\Mmanneg$, we obtain that $a_t$, $(t\in I)$, is fixed on $\Mmanneg$.
\end{proof}

\section{Proof of Theorem~\ref{th:stab1-repr}}\label{sect:proof_th:stab1-repr}
Suppose $\chi(\Mman)<0$ and that $\func:\Mman\to\Psp$ satisfies \AxBd, \AxIsol, and \AxAdm.
We have to find a compact subsurface $\Mmanneg\subset\Mman$ satisfying conditions (1)-(3) of Theorem~\ref{th:stab1-repr}.

\subsection*{Construction of $\Mmanneg$}
Let $\crcomp_1,\ldots,\crcomp_k$ be all the critical components of level-sets of $\func$ whose canonical neighbourhoods $\cannbha{\crcomp_i}$ have negative Euler characteristic: $\chi(\cannbha{\crcomp_i})<0$.
Since $\chi(\Mman)<0$, we have by Lemma\;\ref{lm:crcomp-2disk} that this collection is non-empty.
Denote 
$$
\Kcrcomp=\mathop\cup\limits_{i=1}^{k}\crcomp_i,
$$
For each $i=1,\ldots,k$ choose an atom $\regnbha{i}$ for $\crcomp_i$ in a way described in\;\S\ref{sect:atoms}, and let $\cannbha{i}$ be the corresponding canonical neighbourhood of $\crcomp_i$.
Then we can assume that conditions (i) and (ii) of Lemma\;\ref{lm:properties_of_atoms} hold.
In particular, $\regnbha{i}\cap\regnbha{j}=\cannbha{i}\cap\cannbha{j}=\varnothing$ for $i\not=j$.

Denote $\regnbhNDn := \mathop\cup\limits_{i=1}^{k}\regnbha{i}$.
Let also $B_1,\ldots,B_q$ be all the connected components of $\overline{\Mman\setminus\regnbhNDn}$ such that every $B_i$ is  a cylinder and $\func$ has no critical points in $B_i$.
Put $$\Mmanneg=\regnbhNDn \cup B_1 \cup \cdots \cup B_q.$$
We will show that $\Mmanneg$ satisfies the statement of Theorem~\ref{th:stab1-repr}.

\begin{example}\rm
Let $\Mman$ be an orientable surface of genus $2$ embedded in $\RRR^3$ in a way shown in Figure~\ref{fig:surfX}a) and $\func:\Mman\to\RRR$ be the projection to the vertical line.
Critical components of level-sets of $\func$ whose canonical neighbourhoods have negative Euler characteristic are denoted by $K_1$ and $K_2$.
The corresponding surface $\Mmanneg$ is shown in Figure~\ref{fig:surfX}b).
\begin{figure}[ht]
\begin{tabular}{ccc}
\includegraphics[width=0.2\textwidth]{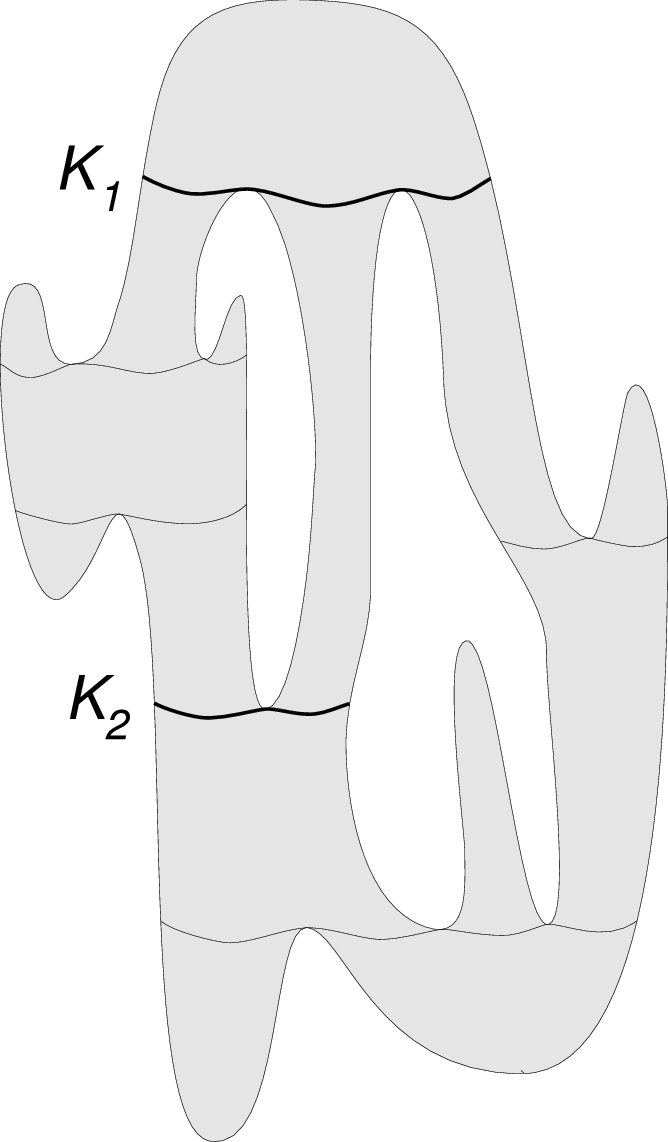}
& \qquad\qquad\qquad &
\includegraphics[width=0.2\textwidth]{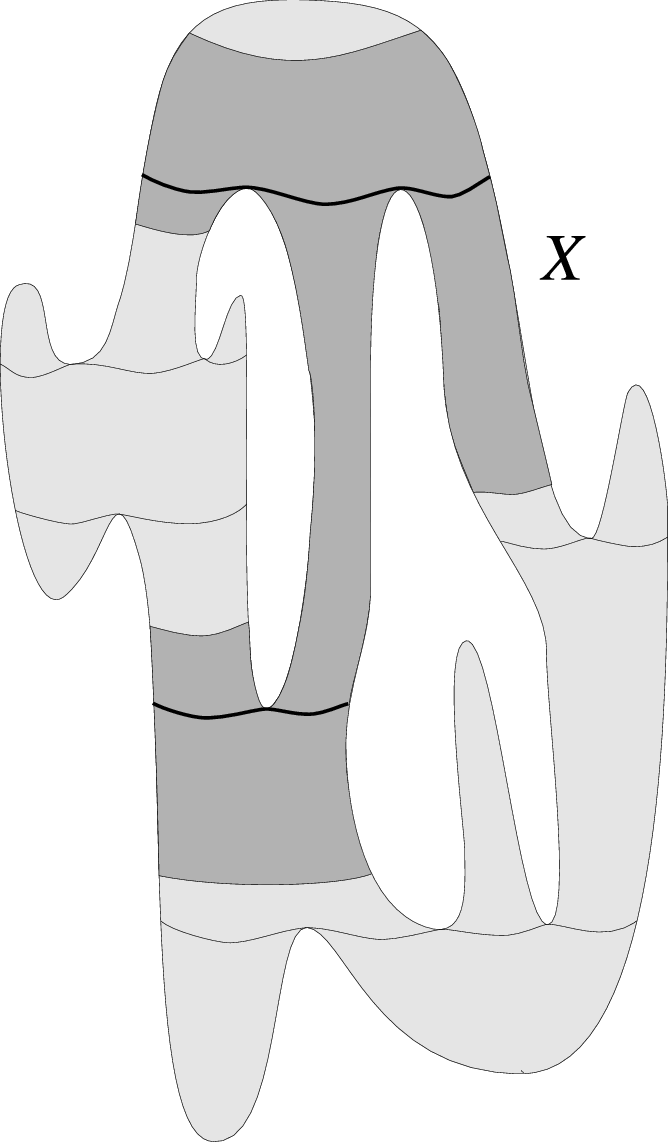}
\\
a) & & b) 
\end{tabular}
\caption{}\protect\label{fig:surfX}
\end{figure}
\end{example}

Before proving Theorem\;\ref{th:stab1-repr} we establish the following statement.

\begin{claim}\label{clm:h_StabAf_pres_om_in_partitf}
{\rm(i)}~Let $\dif\in\StabAf$.
Then $\dif$ preserves every leaf $\omega\subset\regnbhNDn$ of $\partf$ and its orientation.

{\rm(ii)}~Suppose $\dif$ is fixed on a neighbourhood of $\regnbhNDn$.
Then for every connected component $B$ of $\overline{\Mman\setminus\regnbhNDn}$ the restriction $\dif|_{B}$ is isotopic to $\id_{B}$ with respect to a neighbourhood of $\partial B\cap\regnbhNDn$.
\end{claim}
\begin{proof}
(i).~It follows from the definition of $\Kcrcomp$ that $\dif(\Kcrcomp)=\Kcrcomp$.
We claim that in fact $\dif(\crcomp_{i})=\crcomp_{i}$ for all $i=1,\ldots,k$.

Indeed, suppose that $\dif(\crcomp_i)=\crcomp_j$ for some $j$.
Then by Lemma\;\ref{lm:properties_of_atoms} $\dif(\regnbha{i})=\regnbha{j}$ and $\dif(\cannbha{i})=\cannbha{j}$.
On the other hand, since $\cannbha{i}$ is incompressible, $\chi(\cannbha{i})<0$, and $\dif$ is isotopic to $\id_{\Mman}$, it follows from (1) of Proposition\;\ref{pr:incompr-surf-homotopy} that $\dif(\cannbha{i})\cap\cannbha{i}\not=\varnothing$.
But $\cannbha{i}\cap\cannbha{j}=\varnothing$ for $i\not=j$.
Hence $\dif(\cannbha{i})=\cannbha{i}$ for each $i=1,\ldots,k$.

Denote by $\cwpart_{i}$ the corresponding partition of $\cannbha{i}$, see\;\S\ref{sect:cell-auto}.
Since $\dif$ preserves the set of critical points of $\func$, it follows that $\dif$ preserves the set of vertexes of degree $2$ of $\crcomp_i$.
This implies that the restriction of $\dif$ to $\cannbha{i}$ yields a certain automorphism $\dif^{*}$ of the partition $\cwpart_i$.
As $\chi(\cannbha{i})<0$ and $\dif$ is isotopic to $\id_{\Mman}$, we get from Theorem\;\ref{th:aut-surf-graph} that $\dif$ yields a trivial automorphism of $\cwpart_i$.
In particular, each (critical) leaf $\omega$ of $\partf$ in $\crcomp_i$ is $\dif^{+}$-invariant.

Let $\omega\subset\regnbha{i}$ be a regular leaf of $\partf$ and $e\subset\cannbha{i}$ be the corresponding element of $\cwpart_i$ containing $\omega$, so $e$ is either an open $2$-disk or a half-open cylinder $S^1\times(0,1]$.
Then $$\omega\ \ = \ \ e \ \ \cap \ \ \func^{-1}\circ\func(\omega).$$
Notice that $\dif(e)=e$, since $\dif$ is $\cwpart_i$-trivial.
Moreover, $\func\circ\dif=\func$ implies that $\dif\circ\func^{-1}\circ\func(\omega)=\func^{-1}\circ\func(\omega)$, whence $\dif(\omega)=\omega$.
It remains to note that $\dif$ preserves orientation of $\omega$ since it preserves orientation of leaves in $\crcomp_i$.

\medskip 

(ii)~Let $B$ be a connected component of $\overline{\Mman\setminus\regnbhNDn}$.
Then it follows from Corollary~\ref{cor:concomp-M-Rl0} that $B$ is either
\begin{itemize}
\item[\rm(a)] a $2$-disk, or
\item[\rm(b)] a M\"obius band, or
\item[\rm(c)] a cylinder such that one of its boundary components belongs to $\regnbhNDn$ and another one to $\partial\Mman$, or
\item[\rm(d)] a cylinder with $\partial B\subset\regnbhNDn$.
\end{itemize}
If $B$ is of type (a)-(c), then it is well-known that $\dif$ is isotopic to $\id_{B}$ with respect to a neighbourhood of $\partial B\cap\regnbhNDn$.
See~\cite{Alexander:PNAS:1923, Smale:ProcAMS:1959} for the $2$-disk, and\;\cite{Epstein:AM:1966} for the cases (b) and (c).

Let $Q$ be the union of $\regnbhNDn$ with all the components of types (a)-(c).
Then we can assume that $\dif$ is fixed on $Q$.

It also follows that $Q$ is incompressible and every connected component $Q'$ of $Q$ contains some $\cannbha{j}$.
This implies that $\chi(Q')\leq \chi(\cannbha{j})<0$.
Then by (D) of Proposition~\ref{pr:incompr-surf-homotopy} $\dif$ is homotopic to $\id_{\Mman}$ via a homotopy fixed on $Q$.
In particular, the restriction of $\dif$ to every connected component $B$ of type (d) is homotopic in $B$ to $\id_{B}$ relatively $\partial B$.
\end{proof}

Now we can complete Theorem\;\ref{th:stab1-repr}.

(1) It follows from the definition of $\regnbhNDn$ that $\partial\Mmanneg$ consists of some regular leaves of $\partf$, whence $\func$ is locally constant of $\partial\Mmanneg$.
Moreover by Corollary~\ref{cor:concomp-M-Rl0} every connected component $B$ of $\overline{\Mman\setminus\regnbhNDn}$ and therefore of $\overline{\Mman\setminus\Mmanneg}$ is either a $2$-disk, or a cylinder, or a M\"obius band.

It is also easy to see that $B$ contains critical points of $\func$.
Indeed, suppose $B$ is either a $2$-disk or a M\"obius band.
Since $\func$ is constant on $\partial B$, it follows that $\func|_{B}$ is null-homotopic.
Hence $\func$ must have local extremes in $\Int B$.

On the other hand, if $B$ is a cylinder containing no critical points of $\func$, then by the construction of $\Mmanneg$ we should have that $B\subset\Mmanneg$ which is impossible.

\medskip 

Statement (2) is a particular case of (ii) of Claim\;\ref{clm:h_StabAf_pres_om_in_partitf}.

\medskip 

(3) We have to show that the inclusion $i:\StabAfneg\subset \StabAf$ yields a bijection $i_0:\pi_0\StabAfneg\approx\pi_0\StabAf$.

\begin{claim}
The map $i_{0}:\pi_0 \StabAfneg \to \pi_0 \StabAf$ is an epimorphism.
\end{claim}
\begin{proof}
Let $\dif\in\StabAf$.
We have to show that $\dif$ is isotopic in $\StabAf$ to a diffeomorphism fixed on $\Mmanneg$.

By (i) of Claim\;\ref{clm:h_StabAf_pres_om_in_partitf} $\dif$ preserves the foliation of $\partf$ on $\regnbhNDn$.
Hence by Proposition\;\ref{prop:deform-in-Stabf} applied to each critical component $\crcomp_i$, $(i=1,\ldots,k)$, $\dif$ is isotopic in $\StabAf$ to a diffeomorphism fixed on some neighbourhood of $\regnbhNDn$, so we can assume that $\dif$ itself is fixed near $\regnbhNDn$.

Let $B_i$, $(i=1,\ldots,q)$, be a connected component of $\overline{\Mmanneg\setminus\regnbhNDn}$.
By the construction $B_i$ is a cylinder being a union of regular leaves of $\partf$ and containing no critical points of $\func$.
Choose an orientation for $B_i$.
Then we can define a Hamiltonian flow $\HFlow:B_i\times\RRR\to B_i$ of $\func$ on $B_i$ whose orbits are leaves $\partf$ belonging to $B_i$.
Notice that $\dif$ is fixed on some neighbourhood of $\partial B_i\cap\regnbhNDn$ and by (ii) the restriction of $\dif$ to $B$ is homotopic to $\id_{B_i}$ relatively $\partial B_i$.
Then by\;\cite[Lm.\;4.12]{Maks:AGAG:2006} there exists a $\Cinf$ function $\afunc:B_i\to\RRR$ such that $\afunc=0$ on some neighbourhood of $\partial B_i\cap\regnbhNDn$ and $\dif(x)=\HFlow(x,\afunc(x))$ for all $x\in B_i$.

Notice that $\partial B_i\cap\regnbhNDn$ separates $\Mman$.
Then the map
\begin{equation}\label{equ:isotopy_h_to_fix_on_X}
a:\Mman\times I\to\Mman,\qquad
a(x,t) = 
\begin{cases}
\HFld(x,t\afunc(x)), & x \in B_i, \\
\dif(x), & x \in \Mman\setminus B_i
\end{cases}
\end{equation}
is an isotopy of $\dif$ in $\Stabf$ to a diffeomorphism fixed on $B_i$.
Applying this to each $B_i$ we will made $\dif$ fixed on all of $\Mmanneg$.
\end{proof}

\begin{claim}
$i_{0}:\pi_0 \StabAfneg \to \pi_0 \StabAf$ is a monomorphism.
\end{claim}
\begin{proof}
Let $\StabIdAf$ and $\StabIdAfneg$ be the identity path components of $\StabAf$ and $\StabAfneg$ respectively.
It is clear that $\StabIdAf = \StabIdf$.
Hence an injectivity of $i_0$ means that 
$$\StabIdAfneg = \StabAfneg \cap \StabIdAf = \StabAfneg \cap \StabIdf.$$

Evidently, $\StabIdAfneg \subset \StabAfneg \cap \StabIdf$.

Conversely, let $\dif\in\StabAfneg \cap \StabIdf$, so $\dif$ is fixed on some neighbourhood of $\Mmanneg$ and there exists an isotopy $g_t:\Mman\to\Mman$ in $\Stabf$ between $\dif_0=\id_{\Mman}$ and $\dif_1=\dif$.
Then by Proposition\;\ref{prop:deform-homotopies-in-Stabf} this isotopy can be made fixed on some neighbourhood of $\Mmanneg$.
Hence $\dif\in\StabIdAfneg$.
\end{proof}

\def\cprime{$'$}
\providecommand{\bysame}{\leavevmode\hbox to3em{\hrulefill}\thinspace}
\providecommand{\MR}{\relax\ifhmode\unskip\space\fi MR }
\providecommand{\MRhref}[2]{%
  \href{http://www.ams.org/mathscinet-getitem?mr=#1}{#2}
}
\providecommand{\href}[2]{#2}


\begin{thebibliography}{10}

\bibitem{Alexander:PNAS:1923}
J.~W. Alexander, \emph{On the deformation of $n$-cell}, Proc. Nat. Acad. Sci.
  U.S.A. \textbf{9} (1923), no.~12, 406--407.

\bibitem{BolsinovFomenko:1997}
A.~V. Bolsinov and A.~T. Fomenko, \emph{Vvedenie v topologiyu integriruemykh
  gamiltonovykh sistem (introduction to the topology of integrable hamiltonian
  systems)}, ``Nauka'', Moscow, 1997 (Russian). \MR{MR1664068 (2000g:37079)}

\bibitem{Dancer:2:JRAM:1987}
E.~N. Dancer, \emph{Degenerate critical points, homotopy indices and {M}orse
  inequalities. {II}}, J. Reine Angew. Math. \textbf{382} (1987), 145--164.
  \MR{MR921169 (89h:58038)}

\bibitem{EarleEells:DG:1970}
C.~J. Earle and J.~Eells, \emph{A fibre bundle description of teichm\"uller
  theory}, J. Differential Geometry \textbf{3} (1969), 19--43. \MR{MR0276999
  (43 \#2737a)}

\bibitem{EarleSchatz:DG:1970}
C.~J. Earle and A.~Schatz, \emph{Teichm\"uller theory for surfaces with
  boundary}, J. Differential Geometry \textbf{4} (1970), 169--185.
  \MR{MR0277000 (43 \#2737b)}

\bibitem{Epstein:AM:1966}
D.~B.~A. Epstein, \emph{Curves on {$2$}-manifolds and isotopies}, Acta Math.
  \textbf{115} (1966), 83--107. \MR{MR0214087 (35 \#4938)}

\bibitem{Gramain:ASENS:1973}
Andr{\'e} Gramain, \emph{Le type d'homotopie du groupe des diff\'eomorphismes
  d'une surface compacte}, Ann. Sci. \'Ecole Norm. Sup. (4) \textbf{6} (1973),
  53--66. \MR{MR0326773 (48 \#5116)}

\bibitem{JacoShalen:Topology:1977}
William Jaco and Peter~B. Shalen, \emph{Surface homeomorphisms and
  periodicity}, Topology \textbf{16} (1977), no.~4, 347--367. \MR{MR0464239 (57
  \#4173)}

\bibitem{Maks:TA:2003}
Sergiy Maksymenko, \emph{Smooth shifts along trajectories of flows}, Topology
  Appl. \textbf{130} (2003), no.~2, 183--204. \MR{MR1973397 (2005d:37035)}

\bibitem{Maks:AGAG:2006}
\bysame, \emph{Homotopy types of stabilizers and orbits of {M}orse functions on
  surfaces}, Ann. Global Anal. Geom. \textbf{29} (2006), no.~3, 241--285.
  \MR{MR2248072 (2007k:57067)}

\bibitem{Maks:TrMath:2008}
\bysame, \emph{Homotopy dimension of orbits of {M}orse functions on surfaces},
  Travaux Math\'ematiques \textbf{18} (2008), 39--44.

\bibitem{Maks:func-isol-sing}
\bysame, \emph{Functions with isolated singularities on surfaces}, submitted
  (2009), arXiv:math/0806.4704.

\bibitem{Maks:ImSh}
\bysame, \emph{Image of a shift map along the oribts of a flow}, submitted
  (2009), arXiv:math/0902.2418.

\bibitem{ParisRolfsen:AIF:1999}
L.~Paris and D.~Rolfsen, \emph{Geometric subgroups of surface braid groups},
  Ann. Inst. Fourier (Grenoble) \textbf{49} (1999), no.~2, 417--472.
  \MR{MR1697370 (2000f:20059)}

\bibitem{Prishlyak:TA:2002}
A.~O. Prishlyak, \emph{Topological equivalence of smooth functions with
  isolated critical points on a closed surface}, Topology Appl. \textbf{119}
  (2002), no.~3, 257--267. \MR{MR1888671 (2003f:57059)}

\bibitem{Smale:ProcAMS:1959}
Stephen Smale, \emph{Diffeomorphisms of the {$2$}-sphere}, Proc. Amer. Math.
  Soc. \textbf{10} (1959), 621--626. \MR{MR0112149 (22 \#3004)}

\end{thebibliography}
\end{document}